\documentclass[a4paper,twopage,reqno,11pt]{amsart}
\usepackage[top=30mm,right=30mm,bottom=30mm,left=30mm]{geometry}

\usepackage{graphicx}
\usepackage{amsmath}
\usepackage{amssymb}
\usepackage{amsfonts}
\usepackage{amsthm}
\usepackage{amstext}
\usepackage{amsbsy}
\usepackage{amsopn}
\usepackage{amscd}
\usepackage{enumerate}
\usepackage{color}
\usepackage[colorlinks]{hyperref}
\usepackage[hyperpageref]{backref}
\usepackage{multirow}
\usepackage{float}
\usepackage{algorithm2e}
\usepackage{url}
\usepackage{longtable}
\makeatletter
\renewcommand{\@algocf@capt@plain}{above}% formerly {bottom}
\makeatother

\newtheorem{theorem}{Theorem}[section]
\newtheorem{corollary}[theorem]{Corollary}
\newtheorem{lemma}[theorem]{Lemma}
\newtheorem{proposition}[theorem]{Proposition}

\theoremstyle{definition}

\newtheorem{conjecture}[theorem]{Conjecture}
\newtheorem{question}[theorem]{Question}
\newtheorem{remark}[theorem]{Remark}

\numberwithin{equation}{section}

\newcommand{\GL}{\mathrm{GL}}
\newcommand{\SL}{\mathrm{SL}}

\renewcommand{\O}{\mathrm{O}}
\newcommand{\Soc}{\mathrm{Soc}}

\newcommand{\POm}{P \Omega}

\newcommand{\Alt}{Alt}
\newcommand{\Sym}{Sym}

\newcommand{\Aut}{\mathrm{Aut}}

\newcommand{\Out}{\mathrm{Out}}

\newcommand{\rk}{\mathrm{rk}}

\newcommand{\Fbb}{\mathbb{F}}

\newcommand{\Dmc}{\mathcal{D}}
\newcommand{\Bmc}{\mathcal{B}}

\newcommand{\Pmc}{\mathcal{P}}
\newcommand{\Hmc}{\mathcal{H}}

  \def\a{\alpha}

\def\d{\delta}
   
  \def\l{\lambda}

  \def\la{\langle}
  \def\ra{\rangle}

    \def\O{\Omega}

\newcommand{\e}{\epsilon}

\renewcommand{\leq}{\leqslant}
\renewcommand{\geq}{\geqslant}

\newcommand{\imod}[1]{\allowbreak\mkern4mu({\operator@font mod}\,\,#1)}

\begin{document}
 \title{Finite exceptional groups of Lie type and symmetric designs}
%\thanks{This paper is dedicated to Cheryl E Praeger in honor of her inspiring supervision and leadership in group theory and combinatorial design theory.}
 \author[S.H. Alavi]{Seyed Hassan Alavi}%
 \thanks{Corresponding author: S.H. Alavi}
 \address{Seyed Hassan Alavi, Department of Mathematics, Faculty of Science, Bu-Ali Sina University, Hamedan, Iran.
 }%
 \address{Seyed Hassan Alavi, School of Mathematics, Institute for Research in Fundamental Sciences (IPM), P.O. Box: 19395-5746, Tehran, Iran}
 \email{alavi.s.hassan@basu.ac.ir and  alavi.s.hassan@gmail.com (G-mail is preferred)}
 \author[M. Bayat]{Mohsen Bayat}%
 \address{Mohsen Bayat, Department of Mathematics, Faculty of Science, Bu-Ali Sina University, Hamedan, Iran.}%
 \email{m.bayat@sci.basu.ac.ir}
 \author[A. Daneshkhah]{Asharf Daneshkhah}%
 %\thanks{}
 \address{Asharf Daneshkhah, Department of Mathematics, Faculty of Science, Bu-Ali Sina University, Hamedan, Iran.
 }%
  \email{adanesh@basu.ac.ir}

% \dedicatory{Dedicated to Cheryl E. Praeger}%
 %\commby{}%

\maketitle

\begin{abstract}
In this article, we study symmetric $(v, k, \lambda)$ designs admitting a flag-transitive and point-primitive automorphism group $G$ whose socle $X$ is a finite simple exceptional group of Lie type. We prove a reduction theorem, severely restricting the possible parameters of such designs.
We also prove that the parameters $k$ and $\lambda$ are not coprime, and neither of these parameters can be prime. Moreover, if $\lambda$ is at most $100$, we show that there are  two such parameters sets, namely, $(351,126,45)$ and $(378,117,36)$ for $G=X=G_{2}(3)$.
Our analysis depends heavily on detailed information about actions of finite exceptional almost simple groups of Lie type on the cosets of their large maximal subgroups.  In particular, properties derived in the paper about large subgroups and the subdegrees of such actions may be of independent interest.
\end{abstract}

\section{Introduction}\label{sec:intro}

A \emph{symmetric $(v,k,\lambda)$ design} is an incidence structure $\Dmc=(\Pmc,\Bmc)$  consisting of a set $\Pmc$ of $v$ \emph{points} and a set $\Bmc$ of $v$ \emph{blocks} such that every point is incident with exactly $k$ blocks, and every pair of blocks is incident with exactly $\lambda$ points. If $2<k<v-1$, then  $\Dmc$ is called a \emph{nontrivial} symmetric design.
A \emph{flag} of $\Dmc$ is an incident pair $(\alpha,B)$, where $\alpha$ and $B$ are a point and a block of $\Dmc$, respectively. An \emph{automorphism} of a symmetric design $\Dmc$ is a permutation of the points permuting the blocks and preserving the incidence relation. An automorphism group $G$ of $\Dmc$ is called \emph{flag-transitive} if it is transitive on the set of flags of $\Dmc$. If $G$ acts primitively on the point set $\Pmc$, then $G$ is said to be \emph{point-primitive}. We also adopt the standard Lie notation for groups of Lie type as in \cite{b:Atlas,b:KL-90}, for more definitions and notation see Subsection \ref{sec:defn} below. 

A series of interesting results on flag-transitive automorphism groups of symmetric designs suggests investigating symmetric designs admitting point-primitive automorphism groups whose socle $X$ is a non-abelian finite simple group, see for example \cite{a:ADO-2019,a:Zhou-lam100,a:Zieschang-88}. In this direction, possible symmetric designs (up to isomorphism) have  been studied when  $X$ is $A_{n}^{\e}(q)$ with $n\leq 4$ \cite{a:ABD-PSL3,a:ABD-PSL4,a:ABD-PSL2,a:ABDZ-PSU4,a:D-PSU5,a:DZ-PSU3}, ${}^2\!B_2(q)$, ${}^2\!G_2(q)$, ${}^2\!F_4(q)$, ${}^3\!D_4(q)$ \cite{a:Zhou-exp} or a sporadic simple group  \cite{a:Zhou-sym-sporadic}. This paper is devoted to studying symmetric designs admitting a flag-transitive and point-primitive almost simple automorphism group $G$ whose socle is a finite simple exceptional group of Lie type, and our main result is the following theorem.

\begin{theorem}\label{thm:main}
    Let $\Dmc=(\Pmc,\Bmc)$ be a nontrivial symmetric $(v, k, \lambda)$ design with $\lambda\geq 1$ admitting  a flag-transitive and point-primitive automorphism group $G$. Let also $H:=G_{\alpha}$ with $\alpha\in \Pmc$. If $G$ is an almost simple group with socle $X$ a finite simple exceptional group of Lie type, then one of the following holds:
    \begin{enumerate}[{\rm (a)}]
        \item $X=G_{2}(q)$, $H\cap X=\,^{\hat{}}[q^5]:GL_{2}(q)$ is a parabolic subgroup, $v=(q^{6}-1)/(q-1)$, $k=q^5$ and $\lambda=q^4(q-1)$;
        \item $X=G_{2}(q)$, $H\cap X=SL_{3}^{\e}(q):2$ with $\e=\pm$, $v=q^{3}(q^{3}+\e1)/2$,  $k=q^{3}(q^{3}-\e1)/6$,  and $\lambda=q^{3}(q^{3}-\e3)/18$, where $q=3^{a}\geq 3$;
        \item $X=E_{6}(q)$ with $q=p^{a}$,  $H\cap X=\,^{\hat{}}[q^{16}]:D_{5}(q)\cdot (q-1)$ is a parabolic subgroup, and $v=(q^8+q^4+1)(q^9-1)/(q-1)$, $k=m\cdot w(q)+1$ and  $\lambda=q^{-1}(q^{4}+1)^{-1}(m^{2}\cdot w(q)+m)$, where $w(q)=q^{3}+\sum_{i=0}^{11} q^{i}$ and $m<q(q^{4}+1)$.
    \end{enumerate}
\end{theorem}

In Section \ref{sec:example}, we provide some examples of symmetric designs whose automorphism groups are related to finite simple exceptional group of Lie types. In particular, for small $q$,  we know that the designs in part (a) and part (b) of Theorem~\ref{thm:main} do exist and these designs are flag-transitive and point-primitive, see Table~\ref{tbl:cor-main}. Although our further computational evidence  shows that the designs with parameters set in part (c) of Theorem~\ref{thm:main} do not exist, our method introduced in Subsection \ref{sec:method} does not work to rule out this case, and so based on our further observations, we would like to propose the following conjecture:

\begin{conjecture}\label{conj:main}
    If a nontrivial symmetric design $\Dmc$ admits a flag-transitive and point-primitive automorphism group whose socle $X$ is a finite simple exceptional group of Lie type with point-stabiliser $H$, then $\Dmc$ has parameters set as in part (a) or part (b) of Theorem~\ref{thm:main}, $X=G_{2}(q)$ and $H\cap X$ is either $SL_{3}^{\e}(q):2$ with $q=3^{a}\geq 3$ and $\e=\pm$, or $^{\hat{}}[q^5]:GL_{2}(q)$.
\end{conjecture}

A historical problem of determining block designs with their replication numbers $r$ being coprime to $\lambda$ and admitting flag-transitive automorphism groups $G$ reduces to the case where $G$ is a primitive group of almost simple or affine type \cite{a:Zieschang-88}. As an important contribution to this problem, we apply Theorem~\ref{thm:main} and prove in Corollary~\ref{cor:main-1} below that an almost simple automorphism group $G$ with socle a finite simple exceptional group of Lie type does not give rise to a symmetric $(v, k, \lambda)$ design with $\gcd(k,\lambda)=1$. It is worth noting that at the present time we know all possible flag-transitive automorphism groups $G$ of (symmetric) designs with $\gcd(r,\lambda)=1$  excluding the case where $G\leq A\varGamma L_{1}(q)$, see \cite{a:ADM-AS-CP,a:ABD-Un-CP,a:Biliotti-CP-sym-affine,a:Biliotti-CP-nonsol-HA,a:Biliotti-CP-sol-HA}. We moreover prove in Corollary~\ref{cor:main-1} that symmetric designs with $k$ or $\lambda$ prime cannot admit a flag-transitive and point-primitive automorphism groups whose socle is a finite simple exceptional group of Lie type. In general, flag-transitive and point-primitive almost simple automorphism groups of (symmetric) designs with prime replication numbers  or prime $\lambda$ have  been studied, see \cite{a:ABCD-PrimeRep,a:ABD-PrimeLam}.

\begin{corollary}\label{cor:main-1}
    If $\Dmc$ is a nontrivial symmetric $(v, k, \lambda)$ design admitting a flag-transitive and point-primitive automorphism group whose socle is a finite simple exceptional group of Lie type, then $k$ and $\lambda$ are not coprime. Moreover, neither $k$, nor $\lambda$ is prime.
\end{corollary}

We note here that if the socle $X$ of $G$  is a finite classical simple group, then the parameters $k$ and $\lambda$ can be coprime, for example, the Fano plane is a symmetric $(7,3,1)$ design with flag-transitive and point-primitive automorphism group $A_{1}(7)$, and the simple groups $C_{2}(3)$ and $A_{3}(4)$ are flag-transitive and point-primitive automorphism groups of symmetric designs with parameters $(40,13,4)$ and $(85,21,5)$, respectively, see \cite{a:Braic-2500-nopower}.

Symmetric designs with $\lambda$ small have been of most interest. Kantor \cite{a:Kantor-87-Odd} classifies flag-transitive symmetric $(v, k, 1)$ designs (projective planes).  Regueiro~\cite{t:Regueiro}, Zhou and Dong \cite{a:Zhou-lam3-affine} give a complete classification of biplanes ($\lambda=2$) and triplanes ($\lambda=3$) with flag-transitive automorphism groups apart from those admitting a $1$-dimensional affine automorphism group, see also \cite{a:Regueiro-reduction,a:Regueiro-Exp,a:Zhou-lam3-alt,a:Zhou-lam3-excep} and therein references. Note that for $\lambda\leq 3$, there is no flag-transitive and point-primitive nontrivial symmetric design $\Dmc$ whose automorphism group is an almost simple group with socle a finite simple exceptional group of Lie type \cite{a:Regueiro-Exp,a:Saxl2002,a:Zhou-lam3-excep}. As another consequence of Theorem~\ref{thm:main}, we show that there are only two unique such designs for $\lambda\leq 100$.

\begin{corollary}\label{cor:main-2}
    Let $\Dmc$ be a nontrivial symmetric $(v, k, \lambda)$ design with $\lambda \leq 100$. Then $G$ is a flag-transitive and point-primitive automorphism group of $\Dmc$ with socle $X$ a finite simple exceptional group of Lie type if and only if $G=X=G_{2}(3)$, $H\cap X=SL_{3}^{\e}(3):2$ for $\e=\pm$ and $\Dmc$ exists and is as in line $3$ or $4$ of {\rm Table \ref{tbl:cor-main}} and has parameters $(351,126,45)$ or $(378,117,36)$ respectively for $\e=-$ or $\e=+$.
\end{corollary}

In addition to some interesting constructions and examples given in Section~\ref{sec:example}, the designs associated to $G_2(q)$ for $q=2,3,4$ in Table~\ref{tbl:cor-main} have a beautiful geometric description and can be linked to a Cayley algebra. For example, in the case where $G=G_{2}(3)$, the point-stabilisers $SL_{3}^{\e}(q):2$ are stabilisers of plus or minus points of the ``mod 3 Cayley algebra'' and in each case, the point-stabilisers and block-stabilisers are interchanged by an outer automorphism of $G_2(3)$ implying that the designs are self-dual. Therefore, we have the following question:

\begin{question}\label{ques:main}
    Suppose that $\Dmc$ is a symmetric design with a point-primitive and flag-transitive subgroup $G$ of automorphisms which is an almost simple exceptional group of Lie type. Is it true that $G$ is of type $G_2$ and $\Dmc$ is some kind of design on a Cayley algebra?
\end{question}

In the process of proving Theorem \ref{thm:main}, we also prove the following result, Theorem \ref{thm:large-ex}, on large maximal subgroups of almost simple groups with socle a finite simple exceptional group of Lie type, which we believe is of independent interest. In fact, the point-stabiliser $H$ of a flag-transitive automorphism group $G$ of a rank $2$ geometry must be a large subgroup, that is to say, $|G|\leq |H|^{3}$. Moreover, if $G$ is point-primitive, then the subgroup $H$ is also a maximal subgroup of $G$. Alavi and Burness \cite{a:AB-Large-15} study the large subgroups of finite simple groups. We here study large maximal subgroups of almost simple groups $G$ whose socle $X$ is a finite simple exceptional group of Lie type, and then we apply this result to prove our main Theorem \ref{thm:main}.

\begin{theorem}\label{thm:large-ex}
	Let $G$ be a finite almost simple group whose socle $X$ is a finite simple exceptional group of Lie type, and let $H$ be a maximal subgroup of $G$ not containing $X$. If $H$ is a large subgroup of $G$, then $H$ is either parabolic, or one of the subgroups listed in {\rm Table~\ref{tbl:large-exc-nonpar}}.
\end{theorem}

\begin{table}
    \centering
    %\scriptsize
    %\small
    \caption{Large maximal non-parabolic subgroups $H$ of almost simple groups $G$ with  socle $X$ a finite simple exceptional group of Lie type.}\label{tbl:large-exc-nonpar}
    \resizebox{\textwidth}{!}{
        \begin{tabular}{lll}
            \hline\noalign{\smallskip}
            $X$ &
            $H\cap X$ or type of $H$ &
            Conditions \\
            \noalign{\smallskip}\hline\noalign{\smallskip}
            ${}^2\!B_2(q)$ ($q=2^{2n+1} \geq 8$)& $(q+\sqrt{2q}+1){:}4$& $q=8, 32$ \\
            & ${}^2\!B_2(q^{1/3})$ & $q > 8$, $3\mid 2n+1$\\
            ${}^2\!G_2(q)$ ($q=3^{2n+1} \geq 27$) & $A_{1}(q)$\\
            & ${}^2\!G_2(q^{1/3})$ & $3\mid 2n+1$ \\
            %& $2^3{:}7{:}3$, $L_{2}(8)$, $7{:}6$ & $q=3$ \\
            %
            ${}^3\!D_4(q)$ & $A_{1}(q^3)A_{1}(q)$, $(q^2+\e 1q+1)A_{2}^{\e}(q)$, $G_2(q)$ & $\e = \pm$\\
            & ${}^3\!D_4(q^{1/2})$  & $q$ square \\
            & $7^2: SL_{2}(3)$ & $q=2$ \\
            ${}^2\!F_4(q)$ ($q=2^{2n+1} \geq 8$) & ${}^2\!B_{2}(q)\wr 2$, $ B_{2}(q):2$, $ {}^2F_4(q^{1/3})$ &  \\
            & $SU_{3}(q):2$, $ PGU_{3}(q):2$ & $q=8$ \\
            & $A_{2}(3){:}2$, $ A_{1}(25)$, $ \Alt_{6}\cdot 2^{2}$, $5^{2}{:}4 \Alt_{4}$ & $q=2$ \\
            $G_2(q)$ & $A_{2}^{\e}(q)$, $A_{1}(q)^{2}$, $G_2(q^{1/r})$ & $r=2,3$ \\
            %& $A_{2}^{\e}(q)$ & $p=3$ \\
            & ${}^2\!G_2(q)$ & $q=3^{a}$, $a$~is~odd \\
            & $G_2(2)$ & $q=5,7$ \\
            & $A_{1}(13)$, $J_{2}$ & $q =4$ \\
            & $J_{1}$ & $q=11$ \\
            & $2^{3}.A_{2}(2)$ & $q=3, 5$ \\
            $F_4(q)$ & $B_4(q)$, $D_4(q)$, ${}^3\!D_4(q)$ & \\
            & $F_4(q^{1/r})$ & $r=2,3$ \\
            & $A_1(q)C_3(q)$ & $p \ne 2$ \\
            & $C_4(q)$, $C_{2}(q^{2})$, $C_{2}(q)^{2}$ & $p = 2$ \\
            & ${}^2\!F_4(q)$ & $q=2^{2n+1} \geq 2$ \\
            & ${}^3\!D_4(2)$ & $q=3$ \\
            & $Alt_{9-10}$, $A_3(3)$, $J_{2}$& $q = 2$ \\
            & $A_{1}(q)G_{2}(q)$ & $q>3$ odd \\
            & $\Sym_6 \wr \Sym_2$, $F_{4}(2)$& $q = 2$\\
            %& \alert{$A_{4}^{\epsilon}(2)$}& $q = 2$\\
            $E_6^{\e}(q)$ & $A_1(q)A_5^{\e}(q)$, $F_4(q)$ &  \\
            & $(q-\e)D_5^{\e}(q)$ & $\e=-$ \\
            & $C_4(q)$ & $p \neq 2$ \\
            & $E_6^{\pm}(q^{1/2})$ & $\e=+$ \\
            & $E_6^{\e}(q^{1/3})$ &  \\
            & $(q-\e)^2.D_4(q)$ & $(\e,q) \neq (+,2)$ \\
            & $(q^2+\e q+1).{}^3\!D_4(q)$ & $(\e,q) \neq (-,2)$ \\
            & $J_{3}$, $Alt_{12}$, $B_{3}(3)$, $Fi_{22}$ & $(\e,q)=(-,2)$ \\
            %
            %$E_6(q)$ & $A_1(q)A_5(q)$, $F_4(q)$ &  \\
            %& $C_4(q)$ & $p \neq 2$ \\
            %& $E_6^{\varepsilon}(q^{1/r})$ & $(\varepsilon,r)=(+,2),(+,3),(-,2)$ \\
            %& $(q-1)^2.D_4(q)$ & $q \neq 2$ \\
            %& $(q^2+q+1)\cdot {}^3\!D_4(q)$ & \\
            %& $(q-1)D_5(q)$ & \\
            %
            %$E_6^{-}(q)$ & $A_1(q)A_5^{-}(q)$, $F_4(q)$ &  \\
            %& $(q+1)D_5^{-}(q)$ & \\
            %& $C_4(q)$ & $p \neq 2$ \\
            %& $E_6^{-}(q^{1/3})$ & \\
            %& $(q+1)^2.D_4(q)$ & \\
            %& $(q^2- q+1).{}^3\!D_4(q)$ & $q\neq 2$ \\
            %& $J_{3}$, $Alt_{12}$, $B_{3}(3)$, $Fi_{22}$ & $q=2$ \\
            %
            $E_7(q)$ & $(q-\e)E_{6}^{\e}(q)$, $A_1(q)D_6(q)$, $A_7^{\e}(q)$, $A_1(q)F_4(q)$, $E_7(q^{1/r})$ & $\e = \pm$ and $r=2,3$ \\
            &$Fi_{22}$& $q=2$\\
            %&\alert{$A_{6}^{\e}(2)$, $A_{7}^{\e}(2)$, $B_{5}(2)$, $C_{5}(2)$, $D_{5}^{\e}(2)$, $D_{6}^{\e}(2)$} & $q=2$\\
            %
            $E_8(q)$ & $A_1(q)E_7(q)$, $D_8(q)$, $A_2^{\e}(q)E_6^{\e}(q)$, $E_8(q^{1/r})$ & $\e=\pm$ and $r=2,3$ \\
            %& \alert{$A_{8}^{\e}(2)$ , $B_{7}(2)$, $B_{8}(2)$, $C_{7}(2)$, $C_{8}(2)$, $D_{7}^{\e}(2)$, $D_{8}^{\e}(2)$}& $q=2$\\
            %
            \noalign{\smallskip}\hline
        \end{tabular}
    }
\end{table}

\subsection{Outline of proofs}\label{sec:out}

We prove Theorem~\ref{thm:main} and Corollaries \ref{cor:main-1}-\ref{cor:main-2} in Section~\ref{sec:proof}. The symmetric designs with $\lambda\leq 3$ and automorphism groups satisfying the conditions in Theorem~\ref{thm:main} have been studied in \cite{a:Regueiro-Exp,a:Saxl2002,a:Zhou-lam3-excep}, and so we may assume that $\lambda\geq 4$. Since the group $G$ is point-primitive, the point-stabiliser $H$ is maximal in $G$, and flag-transitivity implies that $H$ is large, see Corollary~\ref{cor:large}. We now apply Theorem~\ref{thm:large-ex} and analyse each possible case.  In order to avoid repetition, we describe our method in details in Subsection~\ref{sec:method} and the required information are given in Table~\ref{tbl:rem}. However, in some cases, in addition to applying our explained method in Subsection~\ref{sec:method}, we need some extra argument which will be separately discussed. As a key tool, we frequently apply Lemma~\ref{lem:six}. We also use several important results proved in Section \ref{sec:subdegs} on subdegrees of groups under discussion acting on the right cosets of their maximal subgroups, see Theorem \ref{thm:martin} and Proposition~\ref{prop:g2sl3}.
We moreover use \textsf{GAP} \cite{GAP4}, and apply Lemmas~\ref{lem:comp} and Lemma~\ref{lem:comp}  and Remark~\ref{rem:alg} for computational arguments.

In order to prove Theorem~\ref{thm:large-ex} in Section~\ref{sec:large}, we use the same method as in \cite{a:AB-Large-15}. Note that \cite[Theorem 7]{a:AB-Large-15} allows us only to find large maximal subgroups $H$ of $G$ satisfying $|H\cap X|^{3}< |X|\leq b^{2}\cdot |H\cap X|^{3}$, where $b$ is a divisor of $|\Out(X)|$, see Remark~\ref{rem:large}.
The maximal subgroups of the low-rank groups have been determined, so the proof in these cases is an easy exercise. For the remaining groups,
our starting point here is a reduction theorem of Liebeck and Seitz \cite[Theorem 2]{a:LS-ex-1990}, which essentially allows us to reduce to the case where $H$ is almost simple, with socle $H_0$, say. At this point there are two possibilities, which we consider separately. Write ${\rm Lie}(p)$ for the set of simple groups of Lie type in characteristic $p$, and suppose $G \in {\rm Lie}(p)$ has untwisted Lie rank $n$. If $H_0 \in {\rm Lie}(p)$ has untwisted Lie rank $r$, then the possibilities with $r>n/2$ are given by Liebeck and Seitz \cite{a:LS-MaxLRank05}, but more work is needed to determine the large subgroups with $r \leq n/2$ (an upper bound on $|H|$ given in \cite[Theorem 1.2]{a:LieSha-Probablity-1995} is useful here). Finally, if $H_0 \not\in {\rm Lie}(p)$, then the possibilities for $H$ are determined in \cite{a:LS1999}, and it is straightforward to read off the large examples.

\subsection{Definitions and notation}\label{sec:defn}

Throughout this paper, all groups and incidence structures are finite. 
We here write $\Alt_{n}$ and $\Sym_{n}$ for the alternating group and the symmetric group on $n$ letters, respectively, and we denote by $[n]$ a  group of order $n$. We also adopt the standard Lie notation for groups of Lie type,
for example, we write $A_{n-1}(q)$ and $A_{n-1}^{-}(q)$ in place of $PSL_{n}(q)$ and $PSU_{n}(q)$, respectively, $D_n^{-}(q)$ instead of  $P\Omega_{2n}^{-}(q)$, and $E_6^{-}(q)$ for ${}^2\!E_6(q)$.
% unless making use of Lie notation  leads to confusions in particular when we are dealing with the adjoint group of a certain type.
We may also assume $q>2$ if $G=G_2(q)$ since $G_{2}(2)$ is not simple and $G_2(2)' \cong A^{-}_{2}(3)$. Moreover, we view the Tits group ${}^2\!F_4(2)'$ as a sporadic group. A group $G$ is said to be \emph{almost simple} with socle $X$ if $X\unlhd G\leq \Aut(X)$ where $X$ is a nonabelian simple group.

Recall that a \emph{symmetric $(v,k,\lambda)$ design} is an incidence structure $\Dmc=(\Pmc,\Bmc)$  consisting of a set $\Pmc$ of $v$ \emph{points} and a set $\Bmc$ of $v$ \emph{blocks} such that every point is incident with exactly $k$ blocks, and every pair of blocks is incident with exactly $\lambda$ points. If $2<k<v-1$, then  $\Dmc$ is called a \emph{nontrivial} symmetric design.
A \emph{flag} of $\Dmc$ is an incident pair $(\alpha,B)$, where $\alpha$ and $B$ are a point and a block of $\Dmc$, respectively. An \emph{automorphism} of a symmetric design $\Dmc$ is a permutation of the points permuting the blocks and preserving the incidence relation. An automorphism group $G$ of $\Dmc$ is called \emph{flag-transitive} if it is transitive on the set of flags of $\Dmc$. If $G$ acts primitively on the point set $\Pmc$, then $G$ is said to be \emph{point-primitive}.  Further notation and definitions in both design theory and group theory are standard and can be found, for example, in \cite{b:Atlas,b:Dixon,b:KL-90,b:lander}.

\section{Examples and comments}\label{sec:example}

In this section, we provide some well-known examples of symmetric designs admitting point-primitive automorphism groups.  We also make some relevant comments on Theorem~\ref{thm:main} and Corollary~\ref{cor:main-2}.

In Table~\ref{tbl:cor-main}, we give some examples of the symmetric designs which arise from the study of primitive permutation groups of small degrees, see \cite{a:Braic-2500-nopower,a:Dempwolff2001}. Although the group $G=G_{2}(2)$ (lines 1-2) is not a simple group, it is point-primitive  automorphism group of symmetric $(36,15,6)$ design which is one of the Menon designs. This design is antiflag-transitive and its complement with parameters $(36,21,12)$ is flag-transitive. The symmetric designs admitting almost simple automorphism group with socle $G_{2}(3)$ and $G_{2}(4)$ (lines 3-6) can be constructed in the following general manner. All designs in Table~\ref{tbl:cor-main} do exist and are flag-transitive.

\begin{table}%[h]
	\centering
	%\scriptsize
	\small
	\caption{The parameters when $X=G_{2}(q)$ for  $q=2,3,4$.}\label{tbl:cor-main}
	%    \resizebox{\textwidth}{!}{
	\begin{tabular}{p{.5cm}llllllp{3.6cm}l}
		\hline
		\multicolumn{1}{l}{Line} &
		\multicolumn{1}{l}{$v$} &
		\multicolumn{1}{l}{$k$} &
		\multicolumn{1}{l}{$\lambda$} &
		\multicolumn{1}{l}{$G$} &
		\multicolumn{1}{l}{$X$} &
		\multicolumn{1}{l}{$H\cap X$} &
		Comments &
		References \\ \hline
		\multicolumn{1}{c}{$1$} &
		$36$ &
		$15$ &
		$6$ &
		$G_{2}(2)$ &
		$G_{2}(2)'$ &
		$SL_{3}(2)$&
		Menon design &
		\cite{a:Braic-2500-nopower,a:Dempwolff2001} \\
		\multicolumn{1}{c}{$2$} &
		$63$ &
		$32$ &
		$16$ &
		$G_{2}(2)$ &
		$G_{2}(2)'$ &
		$4{:}S_{4}$&
		Menon design &
		\cite{a:Braic-2500-nopower} \\
		\multicolumn{1}{c}{$3$} &
		$351$ &
		$126$ &
		$45$ &
		$G_{2}(3)$ &
		$G_{2}(3)$ &
		$\SL_{3}^{-}(3){:}2$  &
		Orthogonal design &
		\cite{a:Braic-2500-nopower,a:Dempwolff2001} \\
		\multicolumn{1}{c}{$4$} &
		$378$ &
		$117$ &
		$36$ &
		$G_{2}(3)$ &
		$G_{2}(3)$ &
		$\SL_{3}(3){:}2$ &
		Orthogonal design &
		\cite{a:Braic-2500-nopower,a:Dempwolff2001} \\
		\multicolumn{1}{c}{$5$} &
		$1365$ &
		$1024$ &
		$768$ &
		$G_{2}(4)$ &
		$G_{2}(4)$ &
		$2^{2+8}{:}GL_2(4)$ &
		Complement of $PG_5(4)$ &
		\cite{a:Braic-2500-nopower} \\
		\multicolumn{1}{c}{$6$} &
		$1365$ &
		$1024$ &
		$768$ &
		$G_{2}(4){:}2$ &
		$G_{2}(4)$ &
		$2^{2+8}{:}GL_2(4){:}2$ &
		Complement of $PG_5(4)$ &
		\cite{a:Braic-2500-nopower} \\
		\hline
		{Note:} &\multicolumn{8}{p{13cm}}{\ The last column references a construction of the corresponding design. The group $G_2(2)$ is not simple.}
	\end{tabular}
	%    }
\end{table}

The symmetric designs with parameters set in Theorem~\ref{thm:main}(a) are the complements of symmetric designs with parameters set $((q^{6}-1)/(q-1),(q^{5}-1)/(q-1),(q^{4}-1)/(q-1))$ for $q=p^{a}$ and $p\neq 3$, which is the parameters set of the well-known symmetric designs $\Dmc(\Hmc(q)^{\ast})$ arose from generalized hexagons, see \cite{a:Dempwolff-G2}. A generalized hexagon is a bipartite graph $\Hmc$ of diameter $6$ and girth $12$. We say that $\Hmc$ is of order $(s,t)$ if all vertices of one partition class are of valency $s+1$, and vertices of the other partition class have valency $t+1$. Let $\Hmc=\Hmc(q)$ be a generalized hexagon of order $(q,q)$. Then $\Dmc(\Hmc)$ is a symmetric $((q^{6}-1)/(q-1),(q^{5}-1)/(q-1),(q^{4}-1)/(q-1))$ design with one partition class of vertices of $\Hmc$ as point set $\Pmc$, and blocks of the form $\alpha^{\perp}=\{\beta \in \Pmc \mid d(\alpha,\beta)\leq 4\}$ for $\alpha\in \Pmc$. The only known generalized hexagons of order $(q,q)$ are $\Hmc(q)$ associated with the Chevalley group $G_{2}(q)$ and its dual hexagon $\Hmc(q)^{\ast}$.
If $q$ is odd, then $\Dmc(\Hmc(q))$ is isomorphic to the orthogonal symmetric design of Higman with $d=5$, and if $q$ is even, then $\Dmc(\Hmc(q))$ is isomorphic to $PG(5,q)$ for $q=3^{a}$ and we have $\Dmc(\Hmc(3^{a})^{\ast})=\Dmc(\Hmc(3^{a}))$. For $q=2$ and $4$, we have the symmetric designs with parameters $(63,31,15)$ and $(1365,341,85)$ and rank $4$ antiflag-transitive point-primitive automorphism group $\Aut(G_{2}(q))$ \cite{a:Braic-2500-nopower,a:Dempwolff-G2}. The corresponding complements of these symmetric designs with parameters $(63,32,16)$ and $(1365,1024,768)$ are flag-transitive and point-primitive. These designs arise from Theorem~\ref{thm:main}(a).

The symmetric designs with parameters $v=3^{t}(3^{t}+\e1)/2$, $k=3^{t-1}(3^{t}-\e1)/2$ and $\lambda=3^{t-1}(3^{t-1}-1)/2$ for $t>1$ and $\e=\pm$ can be related to the designs in Theorem \ref{thm:main}(b).  These designs arise from a nondegenerate orthogonal space of dimension $2t+1$ over a finite field $\Fbb_{3}$ with discriminant $(-1)^{t}$. Two symmetric designs with parameters $(351,126,45)$ and $(378,117,36)$ have been constructed in this way respectively for $(t,\e)=(3,-)$ and $(3,+)$. These designs admit a flag-transitive automorphism group $G_{2}(3)$ of rank $3$ and $4$, respectively, see \cite{a:Braic-2500-nopower}.

\section{Preliminaries}\label{sec:pre}

In this section, we state some useful facts in both design theory and group theory. The first lemme  is an elementary result on subgroups of almost simple groups.

\begin{lemma}\label{lem:New}{\rm \cite[Lemma 2.2]{a:ABD-PSL3}}
    Let $G$  be an almost simple group with socle $X$, and let $H$ be maximal in $G$ not containing $X$. Then $G=HX$, and $|H|$ divides $|\Out(X)|\cdot |H\cap X|$.
\end{lemma}

\begin{lemma}\label{lem:Tits}
    Suppose that $\Dmc$ is a symmetric $(v,k,\lambda)$ design admitting a flag-transitive and point-primitive almost simple automorphism group $G$ with socle $X$ of Lie type in characteristic $p$. Suppose also that the point-stabiliser $G_{\alpha}$, not containing $X$, is not a parabolic subgroup of $G$. Then $\gcd(p,v-1)=1$.
\end{lemma}
\begin{proof}
    Note that $G_{\alpha}$ is maximal in $G$, then by Tits' Lemma \cite[1.6]{a:Seitz-TitsLemma}, $p$ divides $|G:G_{\alpha}|=v$, and so  $\gcd(p,v-1)=1$.
\end{proof}

If a group $G$ acts on a set $\Pmc$ and $\alpha\in \Pmc$, the \emph{subdegrees} of $G$ are the size of orbits of the action of the point-stabiliser $G_\alpha$ on $\Pmc$.

\begin{lemma}\label{lem:subdeg}{\rm \cite[3.9]{a:LSS1987}}
    If $X$ is a group of Lie type in characteristic $p$, acting on the set
    of cosets of a maximal parabolic subgroup, and $X$ is not $A_{n-1}(q)$, $D_{n}(q)$
    (with $n$ odd) and $E_{6}(q)$, then there is a unique subdegree which is a power of $p$.
\end{lemma}

\begin{remark}\label{rem:subdeg}
    We remark that even in the cases excluded in Lemma~\ref{lem:subdeg}, many of the maximal parabolic subgroups still have the property as asserted, see proof of  \cite[Lemma 2.6]{a:Saxl2002}. In particular, for an almost simple group $G$ with socle $X=E_{6}(q)$, if $G$  contains a
    graph automorphism or $H =P_{i}$ with $i$ one of $2$ and $4$, then the conclusion of Lemma~\ref{lem:subdeg} is still true.
\end{remark}

\begin{lemma}\label{lem:six}{\rm \cite[Lemma 2.1]{a:ABD-PSL2}}
    Let $\Dmc$ be a symmetric $(v,k,\lambda)$ design, and let $G$ be a flag-transitive automorphism group of $\Dmc$. If $\alpha$ is a point in $\Pmc$ and $H:=G_{\alpha}$, then $v=|G:H|$ and
    \begin{enumerate}[\rm (a)]
        \item $k(k-1)=\lambda(v-1)$;
        %\item $4\lambda(v-1)+1$ is square;
        \item $k\mid |H|$ and $\lambda v<k^2$;
        %\item $k\mid \gcd(\lambda(v-1),|H|)$;
        \item $k\mid \lambda d$, for all nontrivial subdegrees $d$ of $G$.
    \end{enumerate}
\end{lemma}

For a point-stabiliser $H$ of an automorphism group $G$ of a flag-transitive design $\Dmc$, by Lemma~\ref{lem:six}(b), we conclude that $\lambda|G|\leq |H|^{3}$, and so we have that

\begin{corollary}\label{cor:large}
    Let $\Dmc$ be a flag-transitive $(v, k, \lambda)$ symmetric design with automorphism group $G$. Then $|G|\leq |G_{\alpha}|^3$,  where $\alpha$ is a point in $\Dmc$.
\end{corollary}

\begin{lemma}\label{lem:comp}
    Let $\Dmc$ be a symmetric $(v,k,\lambda)$ design with $2k\leq v$ admitting almost simple flag-transitive automorphism group $G$ with socle $X$ and point-stabiliser $H$. Let $mk=\lambda d$, where $d$ is a divisor of $v-1$, for some positive integer $m$. Then the following properties hold:
    \begin{enumerate}[{\rm  (a)}]
        \item $m\mid (k-1)$, and so $\gcd(m,k)=1$;
        \item $\lambda=\lambda_{1}\lambda_{2}$, where $\lambda_{1}=\gcd(\lambda,k-1)$ and $\lambda_{2}=\gcd(\lambda,k)$;
        \item If $k_{1}:=(k-1)/\lambda_{1}$ and $k_{2}:=k/\lambda_{2}$, then $k_{2}$ divides $d$. Moreover, $\lambda_{1}$ divides $m$,
        $\lambda_{1}< k_{2}/2$, $\gcd(k_{1},k_{2})=1$ and $\gcd(\lambda_{1},k_{2})=1$.
    \end{enumerate}
\end{lemma}
\begin{proof}
    \noindent (a) Since $mk=\lambda d $ divides $\lambda(v-1)$ and $\lambda(v-1)=k(k-1)$, it follows that $mk$ divides $k(k-1)$, and so $m$ is a divisor of $k-1$, and hence $m$ is coprime to $k$.\smallskip
    
    \noindent (b) This part follows immediately from  part (a) and the fact that $k(k-1)=\lambda(v-1)$. \smallskip
    
    \noindent (c) Note that $\lambda_{1}$ is relatively prime to $k_{2}$. Since $mk=\lambda d$, it follows that $k_{2}m=\lambda_{1}d$, and so $k_{2}$ divides $d$ and $\lambda_{1}$ is a divisor of $m$. Since also  $\lambda < k$, $mk=\lambda d$ implies that  $m<d$, moreover $k_{2}\neq 1$. Note that $k\leq v/2$. Then $\lambda_{1}=k(k-1)/\lambda_{2}(v-1)\leq k_{2}(v-2)/2(v-1)<k_{2}/2$. The rest is obvious.
\end{proof}

\begin{remark}\label{rem:alg}
 To be precise how this algorithm works, let $G$ be an almost simple group with socle $X$ a finite simple group of Lie type over a finite field of size $q=p^{a}$. Let also $H$ be a maximal subgroup of $G$. Then $v:=|G:H|$, and so  following our arguments in sections below, in particular, following Steps 1-6 in Subsection \ref{sec:method}, at some stage, we obtain some precise possible values for the parameter $v$ for some specific $q=p^{a}$, see for example Table~\ref{tbl:rem}. In these cases, we can obtain $z:=|\Out(X)|\cdot |H\cap X|$, and then compute the greatest common divisor $d$ of $v-1$ and $z$. The input of the algorithm is a list of possible $(v,d,z)$, and, for each divisors $k_{2}\neq 1$ of $d$, we can find $k_{1}=(v-1)/k_{2}$, and for each $\lambda_{1}\leq k_{2}/2$, we obtain $k=1+\lambda_{1}k_{1}$, and then $\lambda=k(k-1)/(v-1)$. We finally check if the parameters $(v,k,\lambda)$ satisfy the conditions in Lemma~\ref{lem:six}, and hence the output is a list of all possible parameters $(v,k,\lambda)$.
\end{remark}

\begin{algorithm}
	\caption{An algorithm based on Lemmas~\ref{lem:six} and~\ref{lem:comp}}\label{alg}
	\small
	\smallskip
	\KwData{A list $L_{1}$ of given parameters $(v,d,z)$ defined in Remark \ref{rem:alg}}
	\KwResult{A list $L$ consits of possible parameters $(v,k,\lambda)$}
	%initialization\;
	$L:=[ \ \ ]$;\\
	\For{$(v,d,z)$ in $L_{1}$}{
		%        $v:=|G:H|$\;
		$L_{2}:=\textsf{DivisorsInt}(d)\setminus\{1\}$\;
		\For{$k_{2}$ in $L_{2}$}{
			$k_{1}:=(v-1)/k_{2}$\;
			$\lambda_{1}:=0$\;
			\While{$\lambda_{1}\leq k_{2}/2$}{
				$\lambda_{1}:=\lambda_{1}+1$\;
				$k:=1+(\lambda_{1}k_{1})$\;
				$\lambda:=\dfrac{k(k-1)}{(v-1)}$\;
				$\lambda_{2}:=\dfrac{k}{k_{2}}$\;
				%$k_{3}:=\dfrac{z}{k}$\;
				\If {$\textsf{IsInt}(\dfrac{z}{k})$ and $\textsf{IsInt}(\lambda)$ and $\textsf{IsInt}(\lambda_{2})$ and $\lambda<k$ and $k\leq \left[\dfrac{v}{2}\right]$}{
					% $m:=\dfrac{\lambda d}{k}$\;
					% \If{$\textsf{IsInt}(m)$ and $\textsf{Gcd}(m,k)=1$}{
					%     $m_{2}:=\dfrac{\lambda_{2}d}{k}$\;
					%     $m_{3}:=\dfrac{k-1}{m}$\;
					%     \If{$\textsf{IsInt}(m_{2})$ and $\textsf{IsInt}(m_{3})$}{
					\textsf{Add} $(L,[v,k,\lambda])$\;
					% }
					%}
				}
			}
		}
	}
	\Return $L$
	\bigskip
\end{algorithm}

\section{Some subdegrees of finite exceptional groups of Lie type}\label{sec:subdegs}

In this section, we prove  Theorem~\ref{thm:martin}, shown to us by Martin Liebeck. This will be useful in reducing the cases we have to consider in the proof of Theorem \ref{thm:main}.

\begin{theorem}\label{thm:martin}
    Let $G$ be an almost simple group with socle $X = X(q)$ an exceptional group of Lie type, and let $H$ be a maximal subgroup of $G$ as in {\rm Table~\ref{tbl:subdegree}}. Then the action of $G$ on the cosets of $H$ has subdegrees dividing $|H:K|$, where $K$ is the subgroup of $H$ listed in the third column of {\rm Table~\ref{tbl:subdegree}}.
\end{theorem}

\begin{table}[h]
    \centering
    \small
    \caption{Some subdegrees of finite exceptional Lie type groups.}\label{tbl:subdegree}
    \begin{tabular}{llll}
        \hline\noalign{\smallskip}
        $X$ &  $H_0$ & $K$ & Conditions\\
        \noalign{\smallskip}\hline\noalign{\smallskip}
        $E_8(q)$ & $A_1(q)E_7(q)$ & $A_1(q)A_1(q)D_6(q)$ \\
        $E_8(q)$& $A_1(q)E_7(q)$  & $E_6^\e(q)$ & $\e =\pm$ \\
        $E_7(q)$ & $A_1(q)D_6(q)$ & $A_5^\e(q)$ & $\e = \pm$ \\
        $E_7(q)$& $A_7^\e(q)$ & $D_4(q)$ & $q$  odd \\
        $E_7(q)$& $A_7^\e(q)$ & $C_4(q)$ & $q>2$  even \\
        $E_7(q)$& $E_6^\e(q) T_1^\e $ & $F_4(q)$  & $(q,\e) \ne (2,-)$ \\
        $E_7(q)$& $E_6^\e(q) T_1^\e $ & $D_5^\e(q)$  & $(q,\e) \ne (2,-)$ \\
        $E_6^\e(q)$ & $A_1(q)A_5^\e(q)$ & $A_2^\e(q)A_2^\e(q)$ \\
        $E_6^\e(q)$ & $A_1(q)A_5^\e(q)$& $A_2(q^2)$ \\
        $E_6^\e(q)$& $D_5^\e T_1^\e $ & $D_4(q)$ \\
        $E_6^\e(q)$&$D_5^\e T_1^\e $&$ A_4^\e(q)$ & $(q,\e) \ne (2,-)$ \\
        $E_6^\e(q)$& $C_4$  & $C_2(q)C_2(q)$& $q$ odd \\
        $E_6^\e(q)$&$C_4$ & $A_3^{\d}(q)$ & $q$ odd, $\d=\pm$ \\
        $F_4(q)$ & $D_4(q)$ & $G_2(q)$ &  $q>2$ \\
        $F_4(q)$ & $D_4(q)$& $A_3^\e(q)$ & $\e = \pm$, $(q,\e) \ne (2,-)$ \\
        $F_4(q)$ & $^3\!D_4(q)$ & $G_2(q)$ & $q>2$  \\
        $F_4(q)$ & $^3\!D_4(q)$ & $A_2^\e(q)$ & $\e=\pm$  \\
        \noalign{\smallskip}\hline
    \end{tabular}
\end{table}

\begin{remark}\label{rem:martin} We offer some comments on the notation used in Table \ref{tbl:subdegree}.
    In all but one case, $H$ is a subgroup of maximal rank in $G$, in the sense of \cite{a:LS-LargeRank-1992},
    and in column 2 of the table, for notational convenience we give a normal subgroup $H_0$ of very small index in $H$;
    the precise structure of $H$ can be found in \cite[Table 5.1]{a:LS-LargeRank-1992}. In the exceptional case, $X = E_6^\e(q)$ and $H_0= C_4(q)$: here $q$ is odd and $H\cap X \cong PSp_8(q){\cdot}2$, the centralizer of a graph automorphism of $X$ (see for example \cite[4.5.1]{b:GLS}). The subgroup $K$ listed in column 3 is a central product of the indicated factors. In the table we have used Lie notation for conciseness. Most, but not all, of the quasisimple factors of $H_0$ and $K$ are of simply connected type. For example the first entry $A_1(q)A_1(q)D_6(q)$ in column 3 is a central product of simply connected groups $SL_2(q)$, $SL_2(q)$ and ${\rm Spin}_{12}^+(q)$; we have chosen not to give such precise information in the table to keep the notation concise, and also because we do not need it in our application of the theorem. Also, $T_1^\e$ denotes a rank 1 torus of order $q-\e$. As a final comment, note that for the entry $A_3^{\d}(q)$ in column 3, $\d = \pm$ for  both possible values of $\e$.
\end{remark}

\noindent {\bf Proof of Theorem \ref{thm:martin}} \quad 
For all but one entry in Table \ref{tbl:subdegree}, we show that
\begin{equation}\label{norm}
    N_H(K) \ne N_G(K).
\end{equation}
Then picking an element $g \in N_G(K)\setminus K$, we have $K \le H \cap H^g$, so that $|H:H\cap H^g|$ divides  $|H:K|$, giving the result. The exceptional entry in Table \ref{tbl:subdegree} is $(X,H_0,K) = (E_6^\e(q),C_4(q),C_2(q)C_2(q))$ which we shall deal with by a separate argument below.

In proving (\ref{norm}) we shall frequently use information about maximal rank subgroups of exceptional groups, to be found mostly in \cite[Table 5.1]{a:LS-LargeRank-1992}.

Consider first $X = E_8(q)$, $H_0=A_1(q)E_7(q)$. When $K = A_1(q)A_1(q)D_6(q)$, the two factors $A_1(q)$ are interchanged by an element of a maximal rank subgroup $D_8(q)$ containing $K$; since such an element cannot lie in $H$, this establishes (\ref{norm}) for this case. When $K = E_6^\e(q)$, this subgroup $K$ is centralized by a subgroup
$A_2^\e(q)$ of $X$, and this is not contained in $H$.

Next consider $X = E_7(q)$. For $H_0=A_1(q)D_6(q)$, a subgroup $K = A_5^\e(q)$ is centralized by a subgroup $A_2^\e(q)$ of $X$, not contained in $H$. Now let $H_0=A_7^\e(q)$. 
%The case where $K = A_5^\e(q)$ follows as in the previous sentence. 
For $q$ odd, let $K$ be a subgroup $D_4(q)$ of $H$. 
Then $N_G(K)$ induces a group $S_3$ of graph automorphisms of $K$ (see \cite[2.15]{a:CLS-Local-92}), and this is not in $H$. And for $q$ even, a subgroup $C_4(q)$ of $H$ centralizes a group of order $q$ in $X$ (see \cite[4.1]{a:LS}), and this can only lie in $H$ when $q=2$. Finally, let $H_0 = E_6^\e(q)T_1^\e$. A subgroup $K = F_4(q)$ of $H$ is centralized in $X$ by $A_1(q)$ (see \cite[4.1]{a:LS}), which for $(q,\e) \ne (2,-)$ does not lie in $H$; the same goes for a subgroup $K = D_5^\e(q)$ of $H$.

Now consider $X = E_6^\e(q)$. For $H_0 = A_1(q)A_5^\e(q)$, the two subgroups $K$ in Table \ref{tbl:subdegree} are centralized in $X$ by a group $A_2^{\pm}(q)$ that does not lie in $H$ (see \cite[Table 5.1]{a:LS-LargeRank-1992}). For
$H_0 = D_5^\e(q)T_1^\e$, a subgroup $K = D_4(q)$ satisfies (\ref{norm}), since $N_G(K)$ induces a group $S_3$ of graph automorphisms of $K$ that does not lie in $H$; and a subgroup $K = A_4^\e(q)$ is centralized in $X$ by $A_1(q)$ which does not lie in $H$ for $(q,\e)\ne (2,-)$.

Finally, suppose $H_0=C_4(q)$ (continuing with $X = E_6^\e(q)$). A subgroup $A_3^{\d}(q)$ of $H$ is centralized in $X$ by  a group $A_1(q)$ not lying in $H$.
Now suppose $K = C_2(q)C_2(q)$, still with $H_0=C_4(q)$. Here we do not prove (\ref{norm}), but establish Theorem \ref{thm:martin} by a different argument as follows. There is an involution $t \in H$ such that $K \le C_H(t)$. Also there is an involution $u$ in the coset of a graph automorphism of $X$, such that $H = C_X(u)$. The restriction of the adjoint module $L(E_6) \downarrow C_4 = L(C_4)+W(\l_4)$, where the latter term is the Weyl module of high weight $\l_4$ (see \cite[p.193]{a:Seitz-91}). Restricting this to $K$, we can compute the eigenvalues of $t,u$ and $tu$ on $L(E_6)$, and we find that $tu$ has fixed point space of dimension 36. Hence $C_X(tu) = C_4(q)$ and $tu$ is $X$-conjugate to $u$. Picking $x \in X$ such that $u^x=tu$, we then have $H\cap H^x = C_X(u,tu) = C_H(t) \ge K$, and so the subdegree $|H:H\cap H^x|$ divides $|H:K|$, as required.

It remains to consider $X = F_4(q)$. For $H_0=D_4(q)$, subgroups $G_2(q)$ and $A_3^\e(q)$ both have centralizer in $X$ containing $A_1(q)$ (see \cite[4.1]{a:LS}), and this is not contained in $H$ for $q>2$ (the $G_2$ case) and for $(q,\e)\ne (2,-)$ (the $A_3^\e$ case). Similarly, for $H_0=\,^3\!D_4(q)$, subgroups $G_2(q)$, $A_2^\e(q)$ have centralizer in $X$ containing $A_1(q)$, $A_2^\e(q)$ respectively, and these do not lie in $H$ (provided $q>2$ in the $G_2$ case). This completes the proof of Theorem \ref{thm:martin}.

\begin{proposition}\label{prop:g2sl3}
    Let $X = G_2(q)$ ($q=p^a$, $p$ prime) and let $H \cong SL_3(q){\cdot}2$ be a maximal subgroup of $X$. Then the subdegrees of $X$ acting on the cosets of $H$ are
    \[
    \begin{array}{l}
    q^2(q^3-1)\;\;(\hbox{multiplicity }\frac{1}{2}(q-3+\d_{p,2})) \\
    \frac{1}{2}q^2(q^3-1)\;\;(\hbox{multiplicity } 1-\d_{p,2}) \\
    (q^3-1)(q^2-1)\;\;(\hbox{multiplicity } 1) \\
    2(q^3-1)\;\;(\hbox{multiplicity } 1).
    \end{array}
    \]
\end{proposition}

\begin{proof} This is almost completely proved in \cite[6.8]{a:LPS}. Let $T = \O_7(q)$ acting on the cosets of $T_\a = N_1^+$, the stabilizer of a hyperplane of $O_7(q)$-space of type $O_6^+$. Then $X = G_2(q) <T$ and $X\cap T_\a = H$, so the action under consideration in Proposition \ref{prop:g2sl3} is contained in the above action of $T$. Lemma 6.8 of \cite{a:LPS} gives the subdegrees of $T$, and shows that $H$ is transitive on all but one of these suborbits, the exception being a suborbit of size $(q^3-1)(q^2+1)$. For $q$ odd, if we let $V = V_7(q)$ be the underlying orthogonal space with quadratic form $Q$, and $\a = \la v\ra$ is the 1-space fixed by $T_\a = N_1^+$, with $Q(v)=1$, then the proof of  \cite[6.8]{a:LPS} shows that the suborbit of size $(q^3-1)(q^2+1)$ in question is
    \[
    \Delta = \{ \langle v+w \rangle \,:\,w\in v^\perp,Q(w)=0\}.
    \]
    The action of $T_\a$ on $\Delta$ is that of $N_1^+$ on the set of nonzero singular vectors in the $O_6^+(q)$-space $v^\perp$, and it is straightforward to see that the subgroup $H = SL_3(q).2$ has two orbits on these, of sizes $2(q^3-1)$ and $(q^3-1)(q^2-1)$, as in the conclusion of the proposition.
    
    For $q$ even, the proof of \cite[Proposition 1]{a:LPS2} again enables us to identify the suborbit $\Delta$ on which $H$ acts intransitively with the set of  nonzero singular vectors in $O_6^+(q)$-space, and again the orbits of $H$ on these are as in the conclusion. This completes the proof.
\end{proof}

\section{Large maximal subgroups of finite exceptional groups of Lie type }\label{sec:large}

Recall that a proper subgroup $H$ of $G$ is said to be \emph{large} if the order of $H$ satisfies the bound $|G|\leq |H|^3 $. In this section, we prove Theorem~\ref{thm:large-ex}. Here we apply the same method as in \cite{a:AB-Large-15}. We will assume $G$ is a finite almost simple group with socle an exceptional group of Lie type. Note that the order of $G$ is given in \cite[Table 5.1.B]{b:KL-90}. We first observe the following elementary lemma:

\begin{lemma}\label{lem:large}
    Let $G$ be a finite almost simple group with socle a non-abelian simple group $X$, and let $H$ be a maximal subgroup of $G$ not containing $X$. Then $H$ is a large subgroup of $G$ if and only if $|X|\leq b^2|H\cap X|^3$, where $b=|G|/|X|$ divides $ |\Out(X)|$.
\end{lemma}
\begin{proof}
    Let $H$ be a maximal large subgroup of $G$ and $b=|G|/|X|$. Since $|H|=b\cdot|H\cap X|$ and $|G|=b\cdot |X|$, it follows that $|X|\leq b^2\cdot |H\cap X|^3$. Conversely, let $b=|G|/|X|$ and $|X|\leq b^2\cdot |H\cap X|^3$. Note that $|X|=|G|/b$ and $|H|=b\cdot|H\cap X|$. Thus $H$ is a large subgroup of $G$.
\end{proof}

\begin{remark}\label{rem:large}
    By Lemma~\ref{lem:large}, to determine the large maximal subgroups $H$ of $G$, we need to verify
    \begin{align}\label{eq:large}
        |X|\leq b^2\cdot |H\cap X|^3,
    \end{align}
    where $b\mid |\Out(X)|$. It is worth noting that such subgroups satisfying $|X|\leq |H\cap X|^{3}$ have been determined in \cite[Theorem 7]{a:AB-Large-15} and we use the same approach as in \cite[Theorem 7]{a:AB-Large-15} to prove Theorem~\ref{thm:large-ex}.
\end{remark}

\begin{proposition}\label{prop:ex1}
    The conclusion of {\rm Theorem~\ref{thm:large-ex}} holds when $X$ is one of the groups $E_{6}^{\e}(2)$, $F_4(2)$, ${}^2\!F_4(q)$, ${}^3\!D_4(q)$, $G_2(q)$, ${}^2\!G_2(q)$ and ${}^2\!B_2(q)$.
\end{proposition}
\begin{proof}
    In each of these cases, the maximal subgroups of $G$ have been determined and the relevant references are listed below (also see \cite[Chapter 4]{b:Wilson}). Note that the list of maximal subgroups of $E_6^{-}(2)$  presented in the Atlas \cite{b:Atlas} is complete (see \cite[p.304]{b:Atlas-Brauer}).
    \begin{center}
        \small
        \begin{tabular}{lccccccc}
            \hline\noalign{\smallskip}
            $G$ & $E_{6}^{\e}(2)$ & $F_4(2)$ & ${}^2\!F_4(q)$ & ${}^3\!D_4(q)$ & $G_2(q)$ & ${}^2\!G_2(q)$ & ${}^2\!B_2(q)$ \\
            \noalign{\smallskip}\hline
            \mbox{Ref.} & \cite{b:Atlas,a:KW-E6-2} & \cite{a:NW-F4-2} & \cite{a:Malle-2F4} &  \cite{a:Kleidman-3D4} &\cite{a:Cooperstein-G2-even,a:Kleidman-G2} & \cite[ p. 61]{a:Kleidman-G2} & \cite{a:Suzuki-64} \\ \hline
        \end{tabular}
    \end{center}
    It is now straightforward to verify Theorem~\ref{thm:large-ex} for these groups. In particular, we note that every maximal parabolic subgroup of $G$ is large.
\end{proof}

Let us now turn our attention to the remaining cases:
\begin{align*}
    F_{4}(q),E_{6}^{\e}(q), E_{7}(q),E_{8}(q),
\end{align*}
where $\e=\pm 1$, $q= p^a$ and $p$ is a prime (and $q>2$ if $X=E_{6}^{\e}(q)$ or $F_{4}(q)$).

Let $\mathfrak{G}$ be a simple adjoint algebraic group of exceptional type over an
algebraically closed field $K$ of characteristic $p>0$, and let $\sigma$ be a surjective endomorphism of $\mathfrak{G}$ such that $X=O^{p'}(\mathfrak{G}_{\sigma})$ is a finite simple group of Lie type. Let $G$ be a finite almost simple group with socle $X$, where $X$ is an exceptional group of Lie type and let $H$ be a maximal subgroup of $G$, not containing $X$. Let also $H_{0}:=\Soc(H\cap X)$. Denote by $Alt_{n}$ and $Sym_{n}$, the alternating and symmetric groups of degree $n$, respectively. We will apply the following reduction theorem of Liebeck and Seitz, see \cite{a:LS03-survey}.

\begin{theorem}\label{thm:ls}
    Let $X=O^{p'}(\mathfrak{G}_{\sigma})$ be a finite exceptional group of Lie type, let $G$ be a group such that $X\leq G\leq \Aut(X)$, and let $H$ be a maximal non-parabolic subgroup of $G$. Then one of the following holds:
    \begin{enumerate}[{\rm \quad (i)}]
        \item $H=N_{G}(\bar{M}_{\sigma})$, where $\bar{M}$ is a $\sigma$-stable closed subgroup of positive dimension in $G$. The possibilities are obtained in {\rm \cite{a:LS-LargeRank-1992,b:LS-Ex-Mem-AMS}}.
        \begin{enumerate}[\rm \quad (a)]
            \item $H$ is reductive of maximal rank (as listed in {\rm Table 5.1} in {\rm \cite{a:LS-LargeRank-1992}}, see also {\rm\cite{b:LS-Ex-Mem-AMS}}).
            \item $\mathfrak{G}=E_{7}$, $p>2$ and $H=({{2}^{2}}\times \POm_{8}(q) \cdot {{2}^{2}})\cdot \Sym_{3}$ or ${}^3\!D_{4}(q).3$.
            \item $\mathfrak{G}=E_{8}$, $p>5$ and $H=PGL_{2}(q)\times \Sym_{5}$.
            \item $H=\bar{M}_{\sigma}$ with $H_{0}=\Soc(H)$ as in {\rm Table \ref{tbl:H}} below.
        \end{enumerate}
        \item $H$ is an exotic local subgroup recorded in {\rm \cite[Table 1]{a:CLS-Local-92}}.
        \item $\mathfrak{G}=E_8$, $p>5$ and $H=(\Alt_{5}\times \Alt_{6}).2^2$.
        \item $H$ is of the same type as $G$ over a subfield of $\Fbb_{q}$ of prime index.
        \item $H$ is almost simple, and not of type (i) or (iv).
    \end{enumerate}
\end{theorem}

\begin{table}[h]
    \centering
    \small
    \caption{Possibilities for $H_{0}$ in Theorem~\ref{thm:ls}(i)(d).}
    \label{tbl:H}
    \begin{tabular}{lll}
        \hline\noalign{\smallskip}
        $X$ & $H_{0}$ & Conditions  \\
        \noalign{\smallskip}\hline\noalign{\smallskip}
        $F_{4}(q)$   & $A_{1}(q)\times G_{2}(q)$ & $p > 2, q >3$ \\
        $E_{6}^{\e}(q)$  & $ A_{2}(q)\times G_{2}(q)$ & \\
        & $A_{2}^{-}(q)\times G_{2}(q)$ & $q >2$\\
        $E_{7}(q)$  & $ A_{1}(q)\times A_{1}(q)$ & $p>3$ \\
        & $A_{1}(q)\times G_{2}(q)$ & $p>2$, $q >3$ \\
        & $A_{1}(q)\times F_{4}(q)$ & $q>3$ \\
        & $ G_{2}(q) \times C_{3}(q)$ & \\
        $E_{8}(q)$  & $A_{1}(q)\times A_{1}^{\e}(q)$ & $p>3$ \\
        & $A_{1}(q)\times G_{2}(q)\times G_{2}(q)$ & $p>2$, $q>3$ \\
        & $G_{2}(q)\times F_{4}(q)$ & $q>3$  \\
        & $ A_{1}(q) \times G_{2}(q^2)$ & $p>2$, $q>3$ \\
        \noalign{\smallskip}\hline
    \end{tabular}
\end{table}

\begin{remark}\label{rem:as}
    Suppose that $H$ is almost simple with socle $H_0$, as in part (v) of Theorem \ref{thm:ls}. Then
    \begin{enumerate}%\addtolength{\itemsep}{0.2\baselineskip}
        \item[(a)] If $H_0 \not\in {\rm Lie}(p)$ then the possibilities for $H_0$ have been determined up to  isomorphism, see \cite[Tables 10.1--10.4]{a:LS1999}.
        \item[(b)] If $H_0 \in {\rm Lie}(p)$  and $\rk(H_{0})>\frac{1}{2}\rk(G)$ with $q=p^{a}>2$. Then by applying \cite[Theorem 3]{a:LST-LargeRankS-1996}, $G=E_{6}^{\e}(q)$ and $H_{0}=C_{4}(q)$ ($q$ odd) or $H_{0}=F_{4}(q)$.
        \item[(c)] If $H_0 \in {\rm Lie}(p)$  and ${\rm rk}(H_0) \leq \frac{1}{2}{\rm rk}(G)$, then $|H|\leq 12aq^{56}$, $4aq^{30}$, $4aq^{28}$ or $4aq^{20}$ accordingly as $G=E_{8}$, $E_{7}$, $E_{6}^{\e}$ or $F_{4}$, see \cite[Theorem 1.2]{a:LieSha-Probablity-1995}. Moreover, if $H_0$ is defined over $\mathbb{F}_{s}$ for some $p$-power $s$, then one of the following holds (see \cite[Theorem 2]{a:Lawther} for the values of $u(G)$ in part (3)):
        \begin{enumerate}\addtolength{\itemsep}{0.2\baselineskip}
            \item[(1)] $s \leq 9$;
            \item[(2)] $H_0 = A_2^{\e}(16)$;
            \item[(3)] $H_0 \in \{A_1(s), {}^2\!B_2(s), {}^2\!G_2(s)\}$ and $s \leq (2,p-1)\cdot u(G)$, where $u(G)$ is defined as follows:
            \begin{center}
                \small
                \begin{tabular}{cccccc}
                    \hline\noalign{\smallskip}
                    $G$ & $G_2$ & $F_4$ & $E_6$ & $E_7$ & $E_8$ \\
                    \noalign{\smallskip}\hline\noalign{\smallskip}
                    $u(G)$ & $12$ & $68$ & $124$ & $388$ & $1312$\\
                    \noalign{\smallskip}\hline
                \end{tabular}
            \end{center}
        \end{enumerate}
    \end{enumerate}
\end{remark}

%Note that if $H\cap X$ is a parabolic subgroup, then it is easy to check that \eqref{eq:large} holds, so for the remainder, we will assume that $H\cap X$ is non-parabolic.

In what follows, we consider possible maximal subgroups $H$ of the almost simple group $G$ with socle $X$ a finite simple exceptional group, and determine whether or not \eqref{eq:large} holds.

\begin{proposition}\label{prop:f4}
    Let $H$ be a maximal non-parabolic subgroup of $X=F_{4}(q)$ with $q=p^{a}>2$. Then $H$ is large  if and only if  $H$ is of type $B_{4}(q)$, $D_4(q)$, ${}^3\!D_4(q)$, $A_1(q)C_3(q)$ ($p \ne 2$), $C_4(q)$ ($p=2$), $C_{2}(q)^{2}$, $C_{2}(q^2)$ ($p=2$), $A_{1}(q)G_{2}(q)$, ${}^2\!F_4(q)$ ($q=2^{a}$, $a$ odd), $F_4(q_0)$ with $q=q_{0}^{r}$ for $r=2,3$, or ${}^3\!D_4(2)$ ($q=3$).
\end{proposition}
\begin{proof}
    By Theorem \ref{thm:ls}, $H$ is of one of the types (i)--(v). We note that $H$ is non-large if $|H|<q^{16}$. So we restrict our attention to the case where $|H|\geq q^{16}$. Here we need only to deal with the subgroups satisfying \eqref{eq:large}.
    
    Suppose $H$ is of type (i). Here $\bar{M}$ is listed in \cite{b:LS-Ex-Mem-AMS}. If $H$ is of maximal rank, then the possibilities for $H$ can be read off from \cite[Table 5.1]{a:LS-LargeRank-1992}. It is now straightforward to check that the only possibilities for $H$ is $A_{1}(q)C_{3}(q)$ with $q$ odd, $C_{4}(q)$, $C_{2}(q)^{2}$ and $C_{2}(q^2)$ with $p=2$, $B_{4}(q)$, $D_{4}(q)$ and $^3D_{4}(q)$. If the socle $H_{0}$ of $H$ is $A_{1}(q)\times G_{2}(q)$ with $p>2$ and $q>3$ (see Table~\ref{tbl:H}), then by \cite[Theorem 5]{a:AB-Large-15}, we must have $G\neq X$.
    
    Clearly $H$ is not of type (iii). Suppose now $H$ is of type (ii). Then $H$ is too small to be large. Let now $H$ be of type (iv), then \eqref{eq:large} holds only for $F_{4}(q^\frac{1}{r})$ with $r=2,3$. Note that the latter case may occur when $G\neq X$.
    
    Suppose $H$ is of type (v). Then $H$ is almost simple but not of type (i) and (iv). Let $H_0$ denote the socle of $H$. If $H_0 \not\in {\rm Lie}(p)$ then the possibilities for $H$ are recorded in \cite[Tables 10.1--10.4]{a:LS1999},
    and it is easy to check that no large examples arise if $q>3$. However, if $q=3$ then $H_0 = {}^3\!D_4(2)$ is a possibility for $X=F_4(3)$.
    Now assume $H_0 \in {\rm Lie}(p)$ and $r={\rm rk}(H_0) \leq 2$. Here \cite[Theorem 1.2]{a:LieSha-Probablity-1995} gives $|H| < 4aq^{20}$, so some additional work is required. There are several cases to consider.
    
    Write $q=p^a$ and $H_0 = X_{r}^{\e}(s)$, where $s=p^b$. First assume $H_0 = A_2(s)$, so $q^{16} \leq |H|<s^{10}$ and thus $b/a \geq 16/10$. By considering the primitive prime divisor $p_{3b}$ of $|H|$ we deduce that $b/a \in \{4,2,8/3\}$. The case $b/a=4$ is ruled out in the proof of \cite[Theorem 1.2]{a:LieSha-Probablity-1995}, and Remark \ref{rem:as}(c) rules out the case $b/a=8/3$. Therefore $H_0 = A_{2}(q^2)$ is the only possibility, and we calculate that $|H|^3<|G|$ unless $q=4$ and $H={\rm Aut}(A_2(16))$. Similarly, if $H_0 = A_2^{-}(s)$ then $H$ is large if and only if $q=4$ and $H={\rm Aut}(A_2^{-}(16))$. However, $A_{2}^{\e}(16)$ is not a subgroup of $F_4(4)$, see proof of \cite[Lemma 5.7]{a:AB-Large-15}.
    
    Next assume $H_0 = C_2(s)$. Here $b/a \geq 17/11$ since $|H|<s^{11}$, and by considering $p_{4b}$ we deduce that $b/a \in \{2,3\}$. The case $b/a=3$ is eliminated in the proof of \cite[Theorem 1.2]{a:LieSha-Probablity-1995}, so we can assume $H_0 = C_2(q^2)$.
    As noted in Remark \ref{rem:as}(ii), such a subgroup is non-maximal if $q>3$, so let us assume $q=3$. Note by \cite[Lemma 5.7]{a:AB-Large-15} that $H_0=C_2(9)$  is not a subgroup of $G=F_4(3)$. The case $H_0 = B_2(s)$ is in a similar manner. Finally, the remaining possibilities for $H_0$ can be ruled out in the usual manner.
\end{proof}

\begin{proposition}\label{prop:e6}
    Let $H$ be a maximal non-parabolic subgroup of $G=E_6^{\e}(q)$, where $q>2$. Then $H$ is large if and only if $H$ is of type $(q-\e)D_{5}^{\e}(q)$ ($\e=-$), $A_1(q)A_5^{\e}(q)$, $F_4(q)$, $(q-\e)^2.D_4(q)$, $(q^2+\e q+1).{}^3\!D_4(q)$, $C_4(q)$ ($p \neq 2$), $E_6^{\delta}(q_{0})$ with $q=q_{0}^{r}$ for $(\e,\delta,r)=(+,+,2)$, $(+,+,3)$, $(+,-,2)$, $(-,-,3)$.
\end{proposition}
\begin{proof}
    If $|H|\leq q^{24}$, then $|H|^3<|G|$, so in this case we may assume that $|H|> q^{24}$. We now apply Theorem \ref{thm:ls}.
    
    If $H$ is of type (i), then  by \cite[Table 5.1]{a:LS-LargeRank-1992}, we have that $H=N_{G}(\bar{M}_{\sigma})$, where $\bar{M}= T_1D_5$, $A_{1}A_{5}$, $T_{2}D_{4}.S_{3}$. If $H$ is of type (ii) then $H=3^6.{\rm SL}_{3}(3)$ (with $p \geq 5$) is the only possibility, and $H$ is non-large. Case (iii) does not apply here. Now assume $H$ is a subfield subgroup of type $E_6^{\delta}(q_0)$ with $q=q_0^r$, then it is easy to see that $q=q_{0}^{r}$ for $(\e,\delta,r)=(+,+,2)$, $(+,+,3)$, $(+,-,2)$, $(-,-,3)$.
    
    Finally, let us assume $H$ is almost simple and not of type (i) or (iv). Let $H_0$ denote the socle of $H$. First assume $H_0 \not\in {\rm Lie}(p)$. Here the possibilities for $H_0$ can be read off from \cite[Tables 10.1--10.4]{a:LS1999} and it is straightforward to check that no large subgroups of this type arise. If $H_0 \in {\rm Lie}(p)$ and ${\rm rk}(H_0) \geq 3$, the subgroups of type $F_{4}$ and $C_{4}$ ($p \neq 2$) are large. Therefore, to complete the proof of the lemma we may assume that $H_0 \in {\rm Lie}(p)$ and ${\rm rk}(H_0) \leq 3$.
    Here \cite[Theorem 1.2(iii)]{a:LieSha-Probablity-1995} gives $|H| \leq 4q^{28}\log_pq$, so some additional work is required.
    
    We proceed as in the proof of \cite[Theorem 1.2]{a:LieSha-Probablity-1995}, using the method described in \cite[Step 3, p.310]{a:LieSax-OrEx87}. Write $q=p^a$ and $H_0 = X_{r}^{\e'}(s)$, where $r={\rm rk}(H_0)$ and $s=p^b$. We consider the various possibilities for $X_r$ (with $r \leq 3$) in turn. Recall that if $c \geq 2$ and $d \geq 3$ are integers (and $(c,d) \neq (2,6)$), then $c_d$ denotes the largest primitive prime divisor of $c^d-1$. Here we may assume by \cite[Lemma .13]{a:BLS} that $|H|\leq q^{32}$.
    
    To illustrate the general approach, consider the case $H_0 = A_3(s)$ with $\e=+$. Here $q^{24} < |H| \leq |{\rm Aut}(H_0)|<s^{17}$, and thus $b/a \geq 25/17$. Now $p_{4b}$ divides $|H|$, and thus $|G|$, so $4b$ divides one of the numbers $6a,8a,9a,12a$, whence $b/a \in \{3,9/4,2,3/2\}$. Moreover, since $p_{3b}$ divides $|G|$ (note that $(p,b) \neq (2,2)$ since we are assuming that $q>2$) we deduce that $b/a \in \{3,2,3/2\}$. However, $H_0 \neq A_3(q^2)$ by the proof of \cite[Theorem 1.2]{a:LieSha-Probablity-1995}, and we have $|H|<q^{24}$ if $H_0 = A_3(q^{3/2})$, and $|H|>q^{32}$ if $H_0 = A_3(q^3)$. This eliminates the case where $H_0 = A_3(s)$. By the same manner, the other cases do not give rise to any large subgroups, see also proof of \cite[Lemma 5.6]{a:AB-Large-15}.
\end{proof}

\begin{proposition}\label{prop:e7}
    Let $H$ be a maximal non-parabolic subgroup of $G=E_7(q)$. Then $H$ is large if and only if $H$ is of type $(q-\e)E_{6}^{\e}(q)$, $A_1(q)D_6(q)$, $A_7^{\e}(q)$, $A_1(q)F_4(q)$, $E_7(q_0)$ with $q=q_0^r$ for $r=2,3$, or $Fi_{22}$ for $q=2$.
\end{proposition}

\begin{proof}
    We proceed as in the proof of the previous proposition. If $|H|<q^{43}$, then $|H|^3<|G|$, so to complete the analysis of this case we may assume that $|H|\geq q^{43}$. We now apply Theorem \ref{thm:ls}.
    
    By inspecting \cite{b:LS-Ex-Mem-AMS} and \cite[Table 1]{a:CLS-Local-92}, it is easy to check that the only examples of type (i) are $N_G(\bar{M}_{\sigma})$ with $\bar{M} = T_1E_6.2$, $A_1D_6.2$, $A_7.2$ or $A_1F_4$. Next suppose $H$ is a subfield subgroup of type $E_7(q_0)$ with $q=q_0^r$. If $r \geq 5$, then $H$ is non-large. The cases (ii) and (iii) do not arise.
    
    To complete the analysis, let us assume $H$ is almost simple, and not of type (i) or (iv). Let $H_0$ denote the socle of $H$. If $H_0 \not\in {\rm Lie}(p)$ then by inspecting \cite[Tables 10.1--10.4]{a:LS1999}, we deduce that $H_{0}$ can be $Fi_{22}$ for $q=2$. Finally, we may assume $H_0 \in {\rm Lie}(p)$ and ${\rm rk}(H_0) \leq 3$ (see Remark \ref{rem:as}(ii)). Here \cite[Theorem 1.2(ii)]{a:LieSha-Probablity-1995} states that $|H| < 4q^{30}\log_pq$, and thus $H$ is non-large.
\end{proof}

\begin{proposition}\label{prop:e8}
    Let $H$ be a maximal non-parabolic subgroup of $G=E_8(q)$. Then $H$ is large if and only if $H$ is of type $A_1(q)E_7(q)$, $D_8(q)$, $A_2^{\e}(q)E_{6}^{\e}(q)$ or $E_8(q_0)$ with $q=q_0^r$ for $r=2,3$.
\end{proposition}
\begin{proof}
    Clearly, if $|H| \leq q^{81}$, then $H$ is non-large, so it remains to consider the maximal subgroups $H$ that satisfy the bounds $|H|>q^{81}$. By Theorem \ref{thm:ls}, $H$ is of type (i)--(v).
    
    If $H$ is of type (i) of maximal rank, then by \cite[Table 5.1]{a:LS-LargeRank-1992}, the only possibilities for $H$ are $D_{8}(q)$, $A_{1}(q)E_{7}(q)$ and $A_{2}^{\e}(q)E_{6}^{\e}(q)$. For other cases in type (i), $|H|$ is in the desired range~\eqref{eq:large}. The possibilities in (ii) are recorded in \cite[Table 1]{a:CLS-Local-92}; either $|H|=2^{15}|{\rm SL}_{5}(2)|$, or $|H| = 5^3|{\rm SL}_{3}(5)|$ (both with $q$ odd). In both cases, $H$ is non-large. Clearly, we can eliminate subgroups of type (iii), and a straightforward calculation shows that a subfield subgroup $H=E_8(q_0)$ with $q=q_0^r$, $r$ prime is large only if $r=2,3$.
    
    Finally, let us assume $H$ is almost simple, and not of type (i) or (iv). Let $H_0$ denote the socle of $H$, and recall that ${\rm Lie}(p)$ is the set of finite simple groups of Lie type in characteristic $p$. First assume that $H_0 \not\in {\rm Lie}(p)$, in which case the  possibilities for $H_0$ are listed in \cite[Tables 10.1--10.4]{a:LS1999}. If $H_0$ is an alternating or sporadic group then $|H| \leq |{\rm Th}|$ and thus $H$ is non-large. Similarly, if $H$ is a group of Lie type then we get $|H| \leq |{\rm PGL}_{4}(5)|2$, and again we deduce that $H$ is non-large. Finally, suppose $H_0 \in {\rm Lie}(p)$. By Remark \ref{rem:as}(c) we have ${\rm rk}(H_0) \leq 4$, so $|H|<12q^{56}\log_pq$ by \cite[Theorem 1.2(i)]{a:LieSha-Probablity-1995}, and thus $H$ is non-large.
\end{proof}

\noindent {\bf Proof of Theorem~\ref{thm:large-ex}}\quad 
The proof follows immediately from Propositions~\ref{prop:ex1} and \ref{prop:f4}-\ref{prop:e8}.

\section{Proof of the main result}\label{sec:proof}
In this section, we prove Theorem~\ref{thm:main} and Corollaries~\ref{cor:main-1} and~\ref{cor:main-2}. Since our arguments in many cases are similar, we here introduce our method in Subsection~\ref{sec:method} below. However, in some cases, we rather prefer to be precise in giving more details of our proof, see for example Proposition~\ref{prop:parab}.

\subsection{Methodology}\label{sec:method}

Suppose that $G$ is a flag-transitive and point-primitive group. Then by Corollary~\ref{cor:large}, the point stabiliser $H:=G_{\alpha}$ is a large maximal subgroup of $G$, where $\alpha$ is a point of $\Dmc$. Therefore, by Theorem~\ref{thm:large-ex}, the subgroup $H$ is a parabolic subgroup or $H\cap X$ is (isomorphic to) one of the subgroups listed in Table~\ref{tbl:large-exc-nonpar}. Moreover, by Lemma~\ref{lem:New},
\begin{align}
    v=\frac{|X|}{|H\cap X|}.\label{eq:v}
\end{align}

For the subgroups $H$ with small values of $q$,  we simply use the  and Remark~\ref{rem:alg} to see if these subgroups  give rise to  possible parameters set $(v,k,\lambda)$. For the remaining subgroups $H$, we first note by Lemma~\ref{lem:six}(b) that $k$ divides $|H|$, and hence  Lemma~\ref{lem:New} implies that
\begin{align}\label{eq:k-g}
    k \text{ divides } b\cdot g(q),
\end{align}
where $b$ is a multiple of $|\Out(X)|$ and $g(q)$ is a polynomial obtained from $|H|$. We note here that in most cases, $g(q)$ is equal to $|H\cap X|$. We know the parameter $v$ by \eqref{eq:v}, and we next apply Lemma~\ref{lem:six}(a) and (b) and conclude that $k$ divides $\lambda f(q)\cdot |\Out(X)|$, where $f(q)$ is a polynomial which is multiple of $\gcd(v-1,g(q))$. If $H$ is not a parabolic subgroup, then in order to obtain $f(q)$, we also use Tits' Lemma~\ref{lem:Tits} saying that $v-1$ is coprime to $p$. Moreover, in the cases where  there are some suitable subdegrees of $G$, we can find a multiple $f(q)$ of the greatest common divisors of $v-1$ and these subdegrees, and then by Lemma~\ref{lem:six}(a) and (c), the parameter $k$ divides $\lambda f(q)$. Therefore, $k$ divides $b_{1}\lambda f(q)$, for some positive integer $b_{1}$ which is mostly  $1$ or a multiple of $|\Out(X)|$. Hence
\begin{align}\label{eq:k-f}
    mk=b_{1}\lambda f(q),
\end{align}
for some positive integer $m$. Since $\lambda <k$, we conclude that
\begin{align}\label{eq:m}
    m<b_{1}\cdot f(q),
\end{align}
Again, by Lemma~\ref{lem:six}(a) and the fact that $mk=b_{1}\lambda f(q)$, we find parameters $k$ and $\lambda$ in terms of $m$, $b_{1}$ and $q$ as below:
\begin{align}\label{eq:k-lam}
    k = \frac{m\cdot (v-1)}{b_{1}\cdot f(q)}+1 \text{ and } \lambda=\frac{m^{2}\cdot (v-1)+mb_{1}\cdot f(q)}{b_{1}^{2}\cdot f(q)^{2}},
\end{align}
Therefore, \eqref{eq:k-g} and \eqref{eq:k-lam} imply that
\begin{align}\label{eq:k-1}
    m\cdot (v-1)+b_{1}\cdot f(q) \text{ divides } b_{1}b\cdot f(q)g(q).
\end{align}
and hence
\begin{align}\label{eq:k-2}
    m\cdot (v-1)+b_{1}\cdot f(q) \leq  b_{1}b\cdot f(q)g(q).
\end{align}
Since $m\geq 1$, we must have
\begin{align}\label{eq:k-3}
    (v-1)+b_{1}\cdot f(q) \leq  b_{1}b\cdot f(q)g(q).
\end{align}
We now check if \eqref{eq:k-3} holds.  Let now this inequality hold for almost all $q$. Suppose that $v-1=v_{1}(q)/c$, where $c$ is a positive integer and $v_{1}(q)$ is a polynomial with integer coefficient in terms $q$.
Then we use Euclidian algorithm and obtain polynomials $h(q)$ and $r(q)$ such that
\begin{align}\label{eq:g-h-d}
    f(q)g(q)=h(q)\cdot v_{1}(q)+r(q),
\end{align}
and then $mb_{1}b\cdot f(q)g(q)=b_{1}bc\cdot  h(q)[m\cdot (v-1)+b_{1}\cdot f(q)]+F(m,q,b,b_{1},c)$, where
\begin{align}\label{eq:F}
    F(m,q,b,b_{1},c)=mb_{1}b\cdot r(q)-b_{1}^{2}bc\cdot f(q)h(q).
\end{align}
We always have to take case of those possible values of $q$ for which $F(m,q,b,b_{1},c)=0$.

Recall by \eqref{eq:k-g} that $k$ divides $b\cdot g(q)$, and so by \eqref{eq:k-lam}, then we conclude that $m\cdot (v-1)+b_{1}\cdot f(q)$ divides $mb_{1}b\cdot f(q)g(q)$, and  if $F(m,q,b,b_{1},c)\neq 0$, then
\begin{align}\label{eq:v-F-1}
    \text{$m\cdot (v-1)+b_{1}\cdot f(q)$ divide $|F(m,q,b,b_{1},c)|$},
\end{align}
and hence, since $m\geq 1$, we have that
\begin{align}\label{eq:v-F-2}
    v< b_{1}b\cdot (|r(q)|+b_{1}c\cdot |f(q)h(q)|).
\end{align}
We then obtain possible $q=p^{a}$ satisfying this inequality, and so for such values of $q$, we can obtain possible parameters set $(v,k,\lambda)$ and check if these parameters give rise to possible designs.

If, in particular, $f(q)$ divides $v_{1}(q)$,  where $v-1=v_{1}(q)/c$, then by taking $w(q)=v_{1}(q)/f(q)$, we have that
\begin{align}\label{eq:w-1}
    m\cdot w(q)+b_{1}c \text{ divides } b_{1}bc\cdot g(q),
\end{align}
and hence
\begin{align}\label{eq:w-2}
    w(q) \leq  b_{1}bc \cdot g(q).
\end{align}
If \eqref{eq:w-2} is true, then we obtain $h(q)$ and $r(q)$ such that
$g(q)=h(q)w(q)+r(q)$, and so $mb_{1}bc \cdot g(q)=b_{1}bc\cdot h(q)[m\cdot w(q)+b_{1}c]+G(m,q,b,b_{1},c)$, where
\begin{align}\label{eq:G}
    G(m,q,b,b_{1},c)=mb_{1}bc\cdot r(q)-b_{1}^{2}b c^{2}\cdot h(q).
\end{align}
Then by the same manner as above, we first consider the possible solutions for $G(m,q,b,b_{1},c)=0$, and next consider the fact that
\begin{align}\label{eq:G-1}
    \text{$mw(q)+b_{1}c$ divides $|G(m,q,b,b_{1},c)|$},
\end{align}
and so we must have
\begin{align}\label{eq:G-2}
    w(q)\leq b_{1}bc\cdot (|r(q)|+b_{1}c\cdot |h(q)|).
\end{align}

We can now summarise our method in the following steps:
\begin{description}%[\quad]
    \item[Step 1] Consider the subgroup $H$ as a parabolic subgroup of $G$ or one of the subgroup listed in Table~\ref{tbl:large-exc-nonpar}, and consider the subgroup $H\cap X$, and then obtain $v$ by \eqref{eq:v}; 
    \item[Step 2] Obtain $|H|$, and determine the polynomial $g(q)$ and the parameter $b$ such that $|H|$ divides $bg(q)$, where $b$ is a multiple of $|\Out(X)|$;
    \item[Step 3] Obtain the polynomial $f(q)$ satisfying \eqref{eq:k-f}. This can be done by determining the greatest common divisors of $v-1$ and $g(q)$. If we have some subdegrees of $G$, we can find a multiple $f(q)$ of the greatest common divisors of $v-1$ and these subdegrees;
    \item[Step 4]  Check the solution of the inequality \eqref{eq:k-3}. If we obtain a finite number of possibilities of $q$ satisfying \eqref{eq:k-3}, then we apply Lemma~\ref{lem:comp}  and Remark~\ref{rem:alg} to obtain possible parameters set $(v,k,\lambda)$;
    \item[Step 5] In the case where, \eqref{eq:k-3} holds for almost all $q$, we obtain $h(q)$ and $r(q)$ satisfying \eqref{eq:g-h-d}, and determine $F(m,q,b,b_{1},c)$ as in \eqref{eq:F}, and then obtain possible $q$ satisfying $F(m,q,b,b_{1},c)=0$ or \eqref{eq:v-F-2}. If $f(q)$ divides $v-1$, then we determine  $G(m,q,b,b_{1},c)$ as in \eqref{eq:G}, and then obtain possible $q$ satisfying $G(m,q,b,b_{1},c)=0$ or \eqref{eq:G-2};
    \item[Step 6] For those values of $q$ obtained in Step 5,  by Lemma~\ref{lem:comp} and Remark~\ref{rem:alg}, we obtain possible parameters set $(v,k,\lambda)$. At this stage, we sometimes obtain $m$ from \eqref{eq:m}, and check if \eqref{eq:v-F-1} or \eqref{eq:G-1}holds.
\end{description}

\subsection{Parabolic, subfield and some numerical cases}\label{sec:prab}

In this section, we deal with the case where $H:=G_{\alpha}$ is a maximal parabolic subgroup or subfield subgroup of $G$ or $H\cap X$ is one of the subgroups listed in Table~\ref{tbl:num}. Note that the cases $(X,H)=(E_{7}(2),A_{7}^{\e}(2))$, $(E_{7}(2),E_{6}^{-}(2))$, $(E_{6}^{-},D_{5}^{-}(2))$, and $(F_{4}(2),D_{4}(2))$ are included in Table~\ref{tbl:num}. In what follows by \cite{a:Zhou-exp}, we only need to consider the case where $X$ is of type $G_{2}$, $F_{4}$, $E_{6}^{\e}$, $E_{7}$ or $E_{8}$.

\begin{proposition}\label{prop:num}
    If $X$ and $H\cap X$ are as in {\rm Table~\ref{tbl:num}}, then there is no symmetric design admitting $G$ as its flag-transitive and point-primitive automorphism group.
\end{proposition}
\begin{proof}
    If $X$ and $H\cap X$  are as in Table~\ref{tbl:num}, then by \eqref{eq:v} and Lemma~\ref{lem:six}, the parameters  $v$ and $k$ are as in the third and fourth columns of Table~\ref{tbl:num}, respectively. For each value of $v$ and $k$, the equality $k(k-1)=\lambda(v-1)$ does not hold for any positive integer $\lambda$.
\end{proof}

\begin{table}%[h]
    \centering
    \scriptsize
    \caption{Parameter $v$ and possible parameter $k$ for some small $q$ in Proposition~\ref{prop:num}}\label{tbl:num}
    \begin{tabular}{llll}
        \hline\noalign{\smallskip}
        \multicolumn{1}{c}{$X$} &
        \multicolumn{1}{c}{$H\cap X$} &
        \multicolumn{1}{c}{$v$} &
        \multicolumn{1}{c}{$k$ divides} \\
        \noalign{\smallskip}\hline\noalign{\smallskip}
        %----------
        %${}^2\!B_2(8)$ & $13{:}4$ & $560$ &  $156$ \\
        %${}^2\!B_2(32)$ & $41{:}4$ & $198400$ & $820$ \\
        %
        %${}^3\!D_4(2)$ & $7^2{:}SL_{2}(3)$ & $179712$ & $3528$ \\
        %
        %${}^2\!F_4(8)$ & $SU_{3}(8){:}2$ & $8004475184742400$ & $99283968$ \\
        %& $PGU_{3}(8){:}2$ & $8004475184742400$ & $99283968$  \\
        %
        $G_2(3)$ & $2^{3}{\cdot}A_{2}(2)$ & $3528$ & $192$ \\
        $G_2(4)$ & $A_{1}(13)$ & $230400$ & $2184$\\
        & $J_{2}$ & $416$ & $1209600$ \\
        $G_2(5)$ & $G_2(2)$ & $484375$ & $12096$ \\
        & $2^{3}{\cdot}A_{2}(2)$ & $4359375$ & $1344$\\
        $G_2(7)$ & $G_2(2)$ & $54925276$ & $12096$ \\
        $G_2(11)$ & $J_{1}$ & $2145199320$ & $175560$\\
        $F_4(2)$  & ${}^3\!D_4(2)$ & $15667200$ & $422682624$\\
        & $D_{4}(2)$ & $3168256$ & $1045094400$ \\
        & $Alt_{9}$ & $18249154560$ & $362880$ \\
        & $Alt_{10}$ & $1824915456$& $3628800$\\
        & $A_3(3){\cdot} 2$ & $272957440$ & $24261120$ \\
        & $J_{2}$ & $2737373184$ & $1209600$ \\
        & $(Sym_6 \wr Sym_2){\cdot}2$ & $3193602048$ & $2073600$ \\
        $E_7(2)$ & $Fi_{22}$ & $123873281581429293827751936$ & $64561751654400$ \\
        & $A_{7}^{+}(2)$ & $373849134340755161088$ & $21392255076846796800$\\
        & $A_{7}^{-}(2)$ & $268914162119825424384$ &$29739884203317657600$  \\
        & $E_{6}^{-}(2)$ & $2488042946297856$& $3214364146718543865446400$\\
        $E_6^{-}(2)$& $J_{3}$ &$253925177425920$ & $301397760$ \\
        & $Alt_{12}$ & $319549996007424$ & $1437004800$ \\
        & $B_{3}(3){:}2$ & $16690645565440$ & $55024220160$ \\
        & $Fi_{22}$ & $1185415168$ & $387370509926400$\\
        & $D_{5}^{-}(2)$ & $1019805696$ & $75046138675200$\\
        \noalign{\smallskip}\hline
    \end{tabular}
\end{table}

\begin{proposition}\label{prop:subf}
    If $H\cap X$ is a subfield subgroup of $X$, then there is no symmetric design with $G$ as its flag-transitive and point-primitive  automorphism group.
\end{proposition}
\begin{proof}
    Let $X:=X(q_{0}^{r})$ and $H\cap X$ be $X(q_{0})$ with $q=q_{0}^{r}$ and $r$ prime. Then
    by Theorem~\ref{thm:large-ex}, $X$ and $H\cap X$ are as in Tables \ref{tbl:subf-r-2} and \ref{tbl:subf-r-3}, and so by \eqref{eq:v}, we obtain $v$ for the corresponding $X$ and $H\cap X$ as in the same tables. We now follow the steps introduced in Subsection~\ref{sec:method} with replacing $q$ by $q_{0}$ in the statements of Steps 1-6. Note by Lemma~\ref{lem:Tits} that $\gcd(v-1,p)=1$. Therefore, for each $(X,H\cap X)$, we find the polynomial $f(q_{0})$. Let $g(q_{0})=|H\cap X|=|X(q_{0})|$ and $b=b_{1}=|\Out(X)|$. Then by \eqref{eq:k-3},  we must have $v-1 \leq f(q_{0})|\Out (X)|^{2}|H\cap X|$, but this inequality does not hold for the possibilities listed in Table~\ref{tbl:subf-r-3} where $q=q_{0}^{3}$. For each remaining cases recorded in Table~\ref{tbl:subf-r-2}, we assume that $h(q_{0})$ is as the same table, and set $r(q_{0})=(v-1)h(q_{0})-f(q_{0})g(q_{0})$. Therefore, for each $(X,H\cap X)$, the inequality \eqref{eq:v-F-2} holds  for $(p,a)$ as in Table \ref{tbl:subf-r-2}. These possible cases can be ruled out by applying Lemma~\ref{lem:comp} and Remark~\ref{rem:alg}.
\end{proof}

\begin{table}%[!htbp]
    \centering
    \caption{The parameters and polynomials for the cases where $X=X(q)$ and $H\cap X=X(q_{0})$ is a subfield subgroup with $q=q_{0}^{3}=p^{3t}$.}\label{tbl:subf-r-3}
    \scriptsize
    \resizebox{\textwidth}{!}{
        \begin{tabular}{p{1.5cm}p{4.2cm}llp{4.2cm}}
            \noalign{\smallskip}\cline{1-2}\cline{4-5}\noalign{\smallskip}
            $X=G_{2}(q)$ & $H\cap X=G_{2}(q_{0})$   & &
            $X=F_{4}(q)$ & $H\cap X=F_{4}(q_{0})$   \\
            \noalign{\smallskip}\cline{1-2}\cline{4-5}\noalign{\smallskip}
            $v$ & $q_{0}^{12} \Phi_{3} \Phi_{6} \Phi_{9} \Phi_{18}$  &&
            $v$ &  $q_{0}^{48} \Phi_{3}^2 \Phi_{6}^2 \Phi_{12} \Phi_{9}^2
            \Phi_{18}^2 \Phi_{24}\Phi_{36}$\\
            $f(q_{0})$ & $(q_{0}^2-1)^{2}$ &&
            $f(q_{0})$ & $\Phi_{1}^4 \Phi_{2}^{4} \Phi_{4}^2\Phi_{12}$  \\
            \noalign{\smallskip}\cline{1-2}\cline{4-5}\noalign{\smallskip}
            $X=E_{6}(q)$ & $H\cap X=E_{6}(q_{0})$  & &
            $X=E_{6}^{-}(q)$ & $H\cap X=E_{6}^{-}(q_{0})$  \\
            \noalign{\smallskip}\cline{1-2}\cline{4-5}\noalign{\smallskip}
            $v$ &  $q_{0}^{72} \Phi_{3}^3\Phi_{6}^2\Phi_{9}^2\Phi_{18}^2
            \Phi_{12}\Phi_{15} \Phi_{24}\Phi_{27}\Phi_{36}$ &&
            $v$ & $q_{0}^{72} \Phi_{3}^2 \Phi_{6}^3\Phi_{9}^2 \Phi_{12}
            \Phi_{18}^2 \Phi_{24} \Phi_{30}\Phi_{36}\Phi_{54}$ \\
            $f(q_{0})$ & $\Phi_{1}^6 \Phi_{2}^4 \Phi_{5} \Phi_{4}^2$ $ \Phi_{8}$ &&
            $f(q_{0})$ &  $\Phi_{1}^4 \Phi_{2}^6 \Phi_{10} \Phi_{4}^2 $ $\Phi_{8}$ \\
            \noalign{\smallskip}\cline{1-2}\cline{4-5}\noalign{\smallskip}
            $X=E_{7}(q)$ & $H\cap X=E_{7}(q_{0})$  & &
            $X=E_{8}(q)$ & $H\cap X=E_{8}(q_{0})$   \\
            \noalign{\smallskip}\cline{1-2}\cline{4-5}\noalign{\smallskip}
            $v$ &$q_{0}^{126} \Phi_{3}^4\Phi_{6}^4\Phi_{9}^2
            \Phi_{12}\Phi_{15}\Phi_{18}^2 \Phi_{27} \Phi_{21}$ $
            \Phi_{24}\Phi_{30}\Phi_{36}\Phi_{42}\Phi_{54}$  &&
            $v$ & $q_{0}^{240} \Phi_{3}^4 \Phi_{6}^4 \Phi_{9}^3 \Phi_{12}^2
            \Phi_{15} \Phi_{18}^3 \Phi_{21} \Phi_{24} \Phi_{27}$ $
            \Phi_{30} \Phi_{36}^2\Phi_{42} \Phi_{45} \Phi_{54}
            \Phi_{60} \Phi_{72} \Phi_{90}$ \\
            $f(q_{0})$ & $\Phi_{1}^7\Phi_{2}^7\Phi_{4}^2\Phi_{5}\Phi_{7}\Phi_{8}\Phi_{10}\Phi_{14}$  &&
            $f(q_{0})$ & $\Phi_{1}^8 \Phi_{2}^8 \Phi_{4}^4 \Phi_{5}^2\Phi_{7}        \Phi_{8}^2\Phi_{10}^2 \Phi_{14} \Phi_{20}$  \\
            \noalign{\smallskip}\cline{1-2}\cline{4-5}\noalign{\smallskip}
            Comments: & \multicolumn{4}{p{10cm}}{$q=q_{0}^{3}=p^{3t}$ and $\Phi_{n}:=\Phi_{n}(q_0)$ is the $n$-th cyclotomic polynomial in terms of $q_0$.}
        \end{tabular}
    }
\end{table}

\begin{table}%[!htbp]
    \centering
    \caption{The parameters and polynomials for the cases where $X=X(q)$ and $H\cap X=X(q_{0})$ is a subfield subgroup  with $q=q_{0}^{2}=p^{2t}$.}\label{tbl:subf-r-2}
    \scriptsize
    \resizebox{\textwidth}{!}{
        \begin{tabular}{lp{11cm}}
            \noalign{\smallskip}\hline\noalign{\smallskip}
            $X=G_{2}(q)$ &
            $H\cap X=G_{2}(q_{0})$ \\
            \noalign{\smallskip}\hline\noalign{\smallskip}
            $v$ & $q_{0}^6\Phi_{4}^{2} \Phi_{12}$ \\
            $f(q_{0})$ & $\gcd(3,q-\e1)^{2}\cdot (4q_{0}^2-1)$ \\
            $h(q_{0})$& $36q_{0}^2-81$  \\
            $q_{0}$  & $2^{t}$ with $t\leq 7$, $3^{t}$ with $t\leq 3$, $5^{t}$ with $t\leq 2$, $p$ with $p=7,\ldots,17$ \\
            \noalign{\smallskip}\hline\noalign{\smallskip}
            $X=F_{4}(q)$ & $H\cap X=F_{4}(q_{0})$ \\
            \noalign{\smallskip}\hline\noalign{\smallskip}
            $v$ & $q_{0}^{24} \Phi_{4}^{2} \Phi_{8}\Phi_{12}
            \Phi_{16}\Phi_{24}$  \\
            $f(q_{0})$ & $\Phi_{1}^{4} \Phi_{2}^{4} \Phi_{3}^{2} \Phi_{6}^{2}$ \\
            $h(q_{0})$& $q_{0}^{16}-4q_{0}^{14}+7q_{0}^{12}-12q_{0}^{10}+22q_{0}^{8}-28q_{0}^{6}+
            31q_{0}^{4}-36q_{0}^{2}+29$ \\
            $q_{0}$  & $2^{t}$ with $t\leq 5$, $3^{t}$ with $t\leq 2$, $p$ with $p=5,7$ \\
            \noalign{\smallskip}\hline\noalign{\smallskip}
            $X=E_{6}(q)$ &
            $H\cap X=E_{6}^{\e}(q_{0})$ with $\e=\pm$  \\
            \noalign{\smallskip}\hline\noalign{\smallskip}
            $v$ & $q_{0}^{36}\Phi_{2}^2\Phi_{4}^2 \Phi_{6}\Phi_{8}
            \Phi_{10}\Phi_{12}\Phi_{16} \Phi_{18} \Phi_{24}$ \\
            $f(q_{0})$ & $\Phi_{1}^6 \Phi_{5} \Phi_{3}^3 \Phi_{9}$  \\
            $h(q_{0})$& $q_{0}^{22}-\e2q_{0}^{21}-q_{0}^{20}+\e2q_{0}^{19}+4q_{0}^{18}
            -\e5q_{0}^{17}-3q_{0}^{16}+\e9q_{0}^{15}+q_{0}^{14}-\e15q_{0}^{13}+2q_{0}^{12}+\e22q_{0}^{11}-14
            q_{0}^{10}-\e25q_{0}^{9}+27q_{0}^{8}+\e27q_{0}^{7}-47q_{0}^{6}-\e13q_{0}^{5}
            +68q_{0}^{4}-\e8q_{0}^{3}-83q_{0}^{2}+\e44q_{0}+81$ \\
            %    $r(q_{0})$ &  $r(q_{0})$ & \\
            $q_{0}$  & $2^{t}$ with $t\leq 20$, $3^{t}$ with $t\leq 9$, $p^{t}$ with $p=5,7$ and $t\leq 6$, $p^{t}$ with $p=11,13$ and $t\leq 4$, $19^{t}$ with $t\leq 3$, $p^{t}$ with $p=17,23, \ldots,113$ and $t\leq 2$, $p$ with $p=127,\ldots,3307$\\
            \noalign{\smallskip}\hline\noalign{\smallskip}
            $X=E_{7}(q)$ & $H\cap X=E_{7}(q_{0})$ \\
            \noalign{\smallskip}\hline\noalign{\smallskip}
            $v$ &  $q_{0}^{63} \Phi_{4}^5\Phi_{8} \Phi_{12}^2\Phi_{16}
            \Phi_{20} \Phi_{24} \Phi_{28} \Phi_{36}$ \\
            $f(q_{0})$ & $3\Phi_{1}^7 \Phi_{2}^7 \Phi_{3}^3 \Phi_{6}^3\Phi_{5}
            \Phi_{10} \Phi_{9} \Phi_{18}$ \\
            %    $g(q_{0})$&
            %    $g(q_{0})$ &  \\
            $h(q_{0})$& $3q_{0}^{46}-15q_{0}^{44}+33q_{0}^{42}-57q_{0}^{40}+102q_{0}^{38}-159
            q_{0}^{36}+222q_{0}^{34}-306q_{0}^{32}+396q_{0}^{30}-516q_{0}^{28}+681
            q_{0}^{26}-858q_{0}^{24}+1092q_{0}^{22}-1383q_{0}^{20}+1704q_{0}^{18}-
            2076q_{0}^{16}+2478q_{0}^{14}-2934q_{0}^{12}+3459q_{0}^{10}-4056q_{0}^{8}
            +4719q_{0}^{6}-5451q_{0}^{4}+6285q_{0}^{2}-7149$\\
            $q_{0}$  &  $2^{t}$ with $t\leq 9$, $3^{t}$ with $t\leq 6$, $5^{t}$ with $t\leq 4$, $7^{t}$ with $t\leq 3$, $p^{t}$ with $p=11,13,17$ and $t\leq 2$, $p$ with $p=19,\ldots,179$ \\
            \noalign{\smallskip}\hline\noalign{\smallskip}
            $X=E_{8}(q)$ & $H\cap X=E_{8}(q_{0})$ \\
            \noalign{\smallskip}\hline\noalign{\smallskip}
            $v$ & $q_{0}^{120} \Phi_{4}^4 \Phi_{8}^2\Phi_{12}^2\Phi_{16}^2
            \Phi_{20}\Phi_{24} \Phi_{28} \Phi_{36}\Phi_{40}
            \Phi_{48} \Phi_{60}$  \\
            $f(q_{0})$ & $\Phi_{1}^8 \Phi_{2}^8 \Phi_{3}^4 \Phi_{5}^2
            \Phi_{6}^4 \Phi_{7} \Phi_{9} \Phi_{14}\Phi_{10}^2
            \Phi_{15} \Phi_{18} \Phi_{30}$ \\
            $h(q_{0})$& $q_{0}^{88}-4q_{0}^{86}+7q_{0}^{84}-10q_{0}^{82}+14q_{0}^{80}
            -15q_{0}^{78}+13q_{0}^{76}-10q_{0}^{74}+9q_{0}^{72}-14q_{0}^{70}+22q_{0}^{68}-32
            q_{0}^{66}+39q_{0}^{64}-36q_{0}^{62}+20q_{0}^{60}+q_{0}^{58}-10q_{0}^{56}+
            4q_{0}^{54}+22q_{0}^{52}-58q_{0}^{50}+79q_{0}^{48}-66q_{0}^{46}+16q_{0}^{44}+42q_{0}^{42}
            -73q_{0}^{40}+46q_{0}^{38}+34q_{0}^{36}-132q_{0}^{34}+186q_{0}^{32}-139q_{0}^{30}+8q_{0}^{28}+140q_{0}^{26}-209q_{0}^{24
            }+136q_{0}^{22}+62q_{0}^{20}-278q_{0}^{18}+365q_{0}^{16}-234q_{0}^{14}
            -71q_{0}^{12}+377q_{0}^{10}-486q_{0}^{8}+290q_{0}^{6}+135q_{0}^{4}-548
            q_{0}^{2}+661$ \\
            $q_{0}$  & $2^{t}$ with $t\leq 7$, $3^{t}$ with $t\leq 3$, $5^{t}$ with $t\leq 2$, $p$ with $p=7,\ldots,17$  \\
            \noalign{\smallskip}\hline\noalign{\smallskip}
            Notation: &  $q=q_{0}^{2}=p^{2t}$ and $\Phi_{n}:=\Phi_{n}(q_0)$ is the $n$-th cyclotomic polynomial in terms of $q_0$.
        \end{tabular}
    }
\end{table}

\begin{table}%[h]
    \centering
    %\scriptsize
    \caption{Some parameters for some parabolic subgroups of almost simple groups with socle $X$.}\label{tbl:parab}
    \resizebox{\textwidth}{!}{
        \begin{tabular}{llllll}
            \hline\noalign{\smallskip}
            \multicolumn{1}{l}{Line} &
            \multicolumn{1}{l}{$X$} &
            \multicolumn{1}{l}{$H\cap X$} &
            \multicolumn{1}{l}{$v$} &
            \multicolumn{1}{l}{$|v-1|_{p}$} &
            %\multicolumn{1}{l}{} &
            \multicolumn{1}{l}{Comments}
            \\
            \noalign{\smallskip}\hline\noalign{\smallskip}
            %$line$ & $X$ & $H\cap X$ & $v$ & $S(q)$& $q$ \\ \hline
            $1$  &
            $G_{2}(q)$&
            $P_{1}$  &
            $\Phi_{1}\Phi_{2}\Phi_{3}\Phi_{6}$  &
            $q$ &
            \\
            $2$ &
            $G_{2}(q)$&
            $P_{1,2}$  &
            $\Phi_{2}^{2}\Phi_{3}\Phi_{6}$ &
            $q$&
            $q=3^{a}$, $H$ contains graph automorphism \\
            $3$ &
            $F_{4}(q)$&
            $P_{1,4}$  &
            $\Phi_{2}^{2}\Phi_{3}^{2}\Phi_{4}\Phi_{6}^{2}\Phi_{8}\Phi_{12}$ &
            $2q$ &
            $q=2^{a}$, $H$ contains graph automorphism \\
            $4$ &
            $F_{4}(q)$&
            $P_{2,3}$  &
            $\Phi_{2}^{2}\Phi_{3}^{2}\Phi_{4}^{2}\Phi_{6}^{2}\Phi_{8}\Phi_{12}$&
            $2q$&
            $q=2^{a}$, $H$ contains graph automorphism \\
            $5$  &
            $E_{6}(q)$&
            $P_{1}$&
            $\Phi_{3}^{2}\Phi_{6}\Phi_{9}\Phi_{12}$  &
            $q$ &
            \\
            $6$  &
            $E_{6}(q)$&
            $P_{3}$  &
            $\Phi_{2}\Phi_{3}^{2}\Phi_{4}\Phi_{6}^{2}
            \Phi_{8}\Phi_{9}\Phi_{12}$   &
            $q$ &
            \\
            $7$  &
            $E_{6}(q)$   &
            $P_{1,6}$ &
            $\Phi_{3}^{2}\Phi_{5}\Phi_{6}\Phi_{8}\Phi_{9}\Phi_{12}$&
            $\gcd(2,p)q$ &
            $H$ contains graph automorphism \\
            $8$  &
            $E_{6}(q)$   &
            $P_{3,5}$ &
            $\Phi_{2}\Phi_{3}^{2}\Phi_{4}^{2}\Phi_{5}\Phi_{6}^{2}\Phi_{8}\Phi_{9}\Phi_{12}$&
            $\gcd(2,p)q$ &
            $H$ contains graph automorphism \\
            \noalign{\smallskip}\hline
            Note: & \multicolumn{5}{p{12cm}}{Here $\Phi_{n}:=\Phi_{n}(q)$ is the $n$-th cyclotomic polynomial with $q=p^{a}$ and $p$ prime.}
        \end{tabular}
    }
\end{table}
%\end{landscape}

In what follows, we write $P_{i}$ to denote a standard maximal parabolic subgroup corresponding to deleting the $i$-th node in the Dynkin diagram of $X$, where we label the Dynkin diagram in the usual way, following \cite[p. 180]{b:KL-90}. We also use $P_{i,j}$ to denote the intersection of appropriate parabolic subgroups of
type $P_{i}$ and $P_{j}$.

\begin{proposition}\label{prop:parab-E6}
    If $X=E_{6}(q)$ and $H$ is a parabolic subgroup of $X$, then $H\cap X$ cannot be $P_{1}$. If $H\cap X=P_{3}$, then $v=(q^8+q^4+1)(q^9-1)/(q-1)$, $k=m\cdot w(q)+1$ and $\lambda=q^{-1}(q^{4}+1)^{-1}(m^{2}\cdot w(q)+m)$, where $w(q)=q^{3}+\sum_{i=0}^{11} q^{i}$ and $m<q(q^{4}+1)$.
\end{proposition}
\begin{proof}
    If $H\cap X$ is $P_{1}$ or $P_{3}$, then the parameter $v$ in each case can be read off from Table~\ref{tbl:parab}. We now analyse each case separately.\smallskip
    
    \noindent \textbf{(i)} Suppose first that $X=E_{6}(q)$ and $H\cap X=P_{1}$. Then $v=(q^{3}+1)(q^{4}+1)(q^{9}-1)(q^{12}-1)/(q-1)(q^{2}-1)$.
    Moreover, by \eqref{eq:k-g}
    \begin{align}\label{eq:out}
        k \mid bg(q),
    \end{align}
    where $b=2a\gcd(3,q-1)$ and $g(q)=q^{36}(q-1)(q^2-1)^2(q^3-1)(q^4-1)(q^5-1)$.
    In this case, it follows from \cite{a:Korableva-E6E7} that $X$ has subdegrees $q(q^5-1)(q^4-1)/(q-1)^2$ and $q^{13}(q^{5}-1)(q-1)/(q-1)^{2}$. Moreover, by Lemma \ref{lem:six}(a) and (b), we have that $k$ is a divisor of $\lambda\gcd(v-1,|H|)$. Therefore, by taking $b_{1}:=b$, it follows from   \eqref{eq:k-lam} that
    \begin{align}\label{eq:k-E6}
        k=\frac{m\cdot (v-1)}{b f(q)}+1,
    \end{align}
    where $f(q)=q(q^{4}+q^{3}+q^{2}+q+1)$ and $m$ is a positive integer.
    We first show that the $p$-part of $k$ is less than $q^6$. Assume the contrary. Then $q^3$ divides $k$, and so by \eqref{eq:k-E6}, $q^3$ must divide $m(q^2+q+1)+b$. Let now $n_{1}$ be a positive integer such that $n_{1}q^3-b=m(q^2+q+1)$. Then
    \begin{align}\label{E6-m-n1}
        m=\frac{n_{1}q^3-b}{q^2+q+1}=n_{1}\cdot (q-1)+\frac{n_{1}-b}{q^2+q+1}.
    \end{align}
    We now consider the following three cases:\smallskip
    
    \noindent \textbf{(i.1)} Let $n_{1}=b$. Then $m=b(q-1)$, and so by \eqref{eq:k-E6}, $k=q^3k_1(q)$, where $k_1(q)=q^{18}-q^{17}+q^{16}+q^{14}+q^{11}+q^{8}-q^{7}-2q^2+q-1$. Then by \eqref{eq:out}, $k_{1}(q)$ must divide $bg(q)$. Since $\gcd(k_1(q), g(q))$ divides $q^4+q^3+q^2+q+1$, it follows that $k_1(q)<b(q^4+q^3+q^2+q+1)$, which is impossible.\smallskip
    
    \noindent \textbf{(i.2)} Let $n_{1}<b$. Since $m$ is integer, by \eqref{E6-m-n1}, $q^2+q+1$ must divide $b-n_{1}$. Then $n_{1}=b-n_{2}\cdot (q^2+q+1)$, for some positive integer $n_{2}$. Since $n_{1}>0$, it follows that $n_{2}(q^2+q+1)<b$, and hence $q^2+q+1<b\leq 6a$, which is impossible.\smallskip
    
    \noindent \textbf{(i.3)} Let $n_{1}>b$. Since $m$ is integer, by \eqref{E6-m-n1}, $q^2+q+1$ must divide $n_{1}-b$. Then $n_{1}=n_{2}(q^2+q+1)+b$, for some positive integer $n_{2}$, and so by \eqref{E6-m-n1}, we have that $m=n_{2}q^3+b(q-1)$. Replacing, $m$ in \eqref{eq:k-E6}, we have that $bk=q^{4}t_{1}(a,n_{2},q)+q^{3}(n_{2}-b)$,  for some polynomial $t_{1}(a,n_{2},q)$. Since $q^{6}$ divide $k$, it follows that $n_{2}-6a=0$ or $q$ divides $|n2-b|$, where $b=2a\cdot \gcd(3,q-1)$. We again consider the following three cases:\smallskip
    
    \noindent \textbf{(i.3.1)} Let $n_{2}=b$. Then $m=b\cdot (q^3+q-1)$, and so by \eqref{eq:k-E6}, $k=q^4k_2(q)$, where $k_2(q)=q^{19}+2q^{17}+3q^{15}+2q^{14}+3q^{13}+3q^{12}+3q^{11}+4q^{10}+4q^9+3q^8+4q^7+2q^6+3q^5+3q^4+q^3+2q^2-q+2$. Then by \eqref{eq:out}, $k_{2}(q)$ must divide $b\cdot g_{1}(q)$, where $g_{1}(q)=g(q)/q^{36}$. It follows that $k_2(q)<b\cdot (q-1)(q^2-1)^2(q^3-1)(q^4-1)(q^5-1)$, which is impossible.\smallskip
    
    \noindent \textbf{(i.3.2)} Let $n_{2}<b$. Since $q^6$ divides $k$, $q$ must divide $b-n_{2}$. Then $q<b=2a\cdot \gcd(3, q-1)$. This inequality holds only for $q=2, 4, 16$. For these values of $q$, by Lemma~\ref{lem:comp} and Remark~\ref{rem:alg}, we cannot find any possible parameters set.\smallskip
    
    \noindent \textbf{(i.3.3)} Let $n_{2}>b$. Since $q^{6}$ divide $k$, it follows that $q$ must divide $n_{2}-b$, and so $n_{2}=n_{3}q+b$, for some positive integer $n_{3}$. Since $m=n_{2}q^3+b\cdot (q-1)$, we have that $m=n_{3}q^{4}+b\cdot (q^{3}+q-1)$, and again by \eqref{eq:k-E6}, we have that $bk=q^{5}t_{2}(a,n_{3},q)+q^{4}(n_{3}+2b)$, for some polynomial $t_{2}(a,n_{3},q)$. Since again $q^{6}$ divides $k$, we obtain $n_{3}=n_{4}q-2b$, for some positive integer $n_{4}$. Therefore, $m=n_{4}q^5-b\cdot (2q^4+q^3+q-1)$ and by \eqref{eq:k-E6}, we have that $bk=q^{6}t_{3}(a,n_{4},q)+q^{5}(n_{4}-3b)$, for some polynomial $t_{3}(a,n_{4},q)$. By the same argument, we conclude that $n_{4}=n_{5}q-3b$, for some positive integer $n_{5}$, and hence  $m=n_{5}q^6+b\cdot (3q^5-2q^4+q^3+q-1)$. Note by \eqref{eq:m} that $m\leq bf(q)$, where $f(q)=q(q^{4}+q^{3}+q^{2}+q+1)$ and $b=2a\gcd(3,q-1)$. Thus $q^6+b(3q^5-2q^4+q^3+q-1)<2bq^{5}$, which is impossible.\smallskip
    
    Therefore, our claim is settled and the $p$-part of $k$ is less than $q^{6}$. Thus \eqref{eq:out} implies that $k$  divides $2ag_{1}(q)$, where $g_{1}(q)=q^{6}(q-1)(q^2-1)^2(q^3-1)(q^4-1)(q^5-1)$. We now apply the method explained in Section~\ref{sec:method}, replace $g_{1}(q)$ with $g(q)$, and taking $h(q)=q^3-q^2-3q+1$ and $r(q)=f(q)g_{1}(q)-h(q)(v-1)$, we conclude by \eqref{eq:v-F-2} that $q=p^{a}$ is as below
    \begin{center}
        \small
        \begin{tabular}{l|cccccc}
            $p$ &
            $2$ &
            $3$ &
            $5$ &
            $7$ &
            $11, \ldots, 19$ &
            $31, \ldots, 103$ \\
            \hline
            $a\leq$
            & $14$
            & $5$
            & $4$
            & $3$
            & $2$
            & $1$ \\
        \end{tabular}
    \end{center}
    \noindent For these values of $q$, by Lemma~\ref{lem:comp} and Remark~\ref{rem:alg}, we cannot find any possible parameters set.\smallskip
    
    \noindent \textbf{(ii)} Suppose that $X=E_{6}(q)$ and $H\cap X=P_{3}$. Here $v=(q^8+q^4+1)(q^9-1)/(q-1)$. Note by \cite[p. 345]{a:Saxl2002} that $G$ has nontrivial subdegrees  $q(q^{8}-1)(q^{3}+1)/(q-1)$ and $q^{8}(q^{5}-1)(q^{4}+1)/(q-1)$, and so by Lemma~\ref{lem:six}(c), we conclude that $k$ divides $\lambda f(q)$, where $f(q)=q(q^{4}+1)$. Then by \eqref{eq:k-lam} and \eqref{eq:m}, we have that $ k=m\cdot w(q)+1$ and $\lambda=q^{-1}(q^{4}+1)^{-1}(m^{2}\cdot w(q)+m)$,
    where $w(q)=q^{3}+\sum_{i=0}^{11} q^{i}$ and $m<q(q^{4}+1)$.
\end{proof}

\begin{proposition}\label{prop:parab}
    Let $H$ be a maximal parabolic subgroup of $G$ and $(X,H\cap X)\neq (E_{6}(q), P_{i})$, for $i=1,3$. Then $X=G_{2}(q)$ and $H\cap X=\,^{\hat{}}[q^5]:\GL_{2}(q)$ and $v=q^5+q^4+q^3+q^2+q+1$, $k=q^5$ and  $\lambda=q^5-q^4$.
    %one of the following holds:
    %\begin{enumerate}[\rm (a)]
    %    \item $X=G_{2}(q)$ and $H\cap X=\,^{\hat{}}[q^5]:\GL_{2}(q)$ and $v=q^5+q^4+q^3+q^2+q+1$ , $k=q^5$ and  $\lambda=q^5-q^4$;
    %    \item $X=E_{6}(q)$ and  $H\cap X=P_{3}$, and  $v=(q^8+q^4+1)(q^9-1)/(q-1)$, $k=m\cdot w(q)+1$ and $\lambda=q^{-2}(q^{4}+1)^{-1}(m^{2}w(q)+m$, where $w(q)=q^{3}+\sum_{i=0}^{11} q^{i}$ and $m<q(q^{4}+1)$.
    %\end{enumerate}
\end{proposition}
\begin{proof}
    Recall by \cite{a:Zhou-exp} that we only need to consider the case where $X$ is of type $G_{2}$, $F_{4}$, $E_{6}^{\e}$, $E_{7}$ or $E_{8}$.
    
    We first consider the possibilities recorded in Table~\ref{tbl:parab}. The case where $(X,H\cap X)=(E_{6}(q), P_{i})$ with $i=1,3$ as in lines $5$ and $6$ has been treated in By Proposition~\ref{prop:parab-E6}. In the remaining cases, we prove that line 1 is the only possible case.
    
    Let first  $X=G_{2}(q)$ and $H\cap X=P_{1}$ as in line $1$ of Table \ref{tbl:parab}. Here $v=q^5+q^4+q^3+q^2+q+1$, and so $|v-1|_{p}=q$,  and by \eqref{eq:k-lam}, we have that
    \begin{align}
        k= m\cdot (q^4+q^3+q^2+q+1)+1 \quad \text{and} \quad
        \lambda =  m^2\cdot (q^3+q^2+q+1)+\frac{m^2+m}{q}.\label{G2-parabolic-subcase1-lam}
    \end{align}
    So by \eqref{eq:m} and the fact that ,  we conclude that $m=q-1$. Therefore, \eqref{G2-parabolic-subcase1-lam} implies that $k=q^5$ and $\lambda=q^5-q^4$. This is part (a).
    
    Let $X=G_{2}(q)$ and $H\cap X=P_{1,2}$ as in line $2$ of Table \ref{tbl:parab}. Here we also have $|v-1|_{p}=q$, and by the same argument as above, since $m<q$ and $q$ divides $m(2m+1)$, we conclude that $m$ is $(q-1)/2$ or $(2q-1)/2$. In the latter case, we have that $2k=q^{5}+2\sum_{i=1}^{6}q^{i}$, which is impossible as $q=3^{a}$. If $m=(q-1)/2$, then $k=q^{5}(q+1)/2$, and so $(q+1)/2$ divides $k$. Note by Lemma~\ref{lem:six}(b) that $k$ divides $2aq^{6}(q-1)^{2}$. Then $q+1$ divides $16a$, and so $q=3$ for which we have the parameters $(1456,486,162)$, but by \cite[p. 473]{a:Braic-2500-nopower}, there is no flag-transitive or antiflag-transitive design with this parameter set.
    
    We now consider the case where the subgroup $H$ contains a graph automorphism listed in lines 3-4 and 7-8 of Table~\ref{tbl:parab}. We here note by Lemma~\ref{lem:subdeg} and Remark~\ref{rem:subdeg} that in these cases we still have a prime power subdegree $p^{t}$. Therefore, in each case, as $|v-1|_{p}$ divides $2q$, it follows from Lemma~\ref{lem:six} that  $k$ divides $\lambda f(q)$ where $f(q)=2q$. We give our argument for $X=F_{4}(q)$ and $H\cap X=P_{1,4}$, and other cases can be ruled out in a similar manner. In this case, we have that $v=(q^6-1)(q^8-1)(q^{12}-1)(q-1)^{-2}(q^4-1)^{-1}$, and so, as noted above, $mk=b_{1}\lambda f(q)$, where $b_{1}=2$, $f(q)=q$ and $m$ is a positive integer. By \eqref{eq:k-lam}, we have that
    \begin{align}
        2k &= m\cdot w(q)+2 \label{eq:F4-k-1} \text{ and }\\
        4\lambda&=m^{2}\cdot t(q)+\frac{2m(m+1)}{q},\label{eq:F4-lam-1}
    \end{align}
    where $w(q)=(v-1)/f(q)$ and $t(q)$ is a polynomial in terms of $q$. Thus \eqref{eq:F4-lam-1} implies that $q$ divides $2m(m+1)$, and hence $q$ divides $2m$ or $2m+2$. If $q$ divides $2m$, then since $m\leq 2q$ by \eqref{eq:m}, it follows that $m=iq/2$ with $i=1,2,3$. By replacing $m$ in \eqref{eq:F4-k-1},  we have that $4k=qt_{i}(q)+4$ for $i=1,2,3$, where $t_{i}(q)$ is a polynomial in terms of $q$. This shows that if $p\mid k$, then $p$-part of $k$ must divide $4$, and hence the $p$-part of $k$ is at most $q^{2}$. Therefore, by \eqref{eq:k-g}, we conclude that $k$ must divide $2ag_{1}(q)$, where $g_{1}(q)=q^{2}(q-1)^2(q^2-1)(q^4-1)$. Therefore, by \eqref{eq:k-3}, we must have $v<4af(q)g_{1}(q)$, and hence $v<4aq^{3}(q-1)^2(q^2-1)(q^4-1)$, which is impossible.
    
    We finally consider those parabolic subgroups which are not listed as in Table~\ref{tbl:parab}. Then $|v-1|_{p}=q$ and $v-1=q\cdot w(q)$, for some polynomial $w(q)$ coprime to $p$. By Lemma~\ref{lem:subdeg} and Remark~\ref{rem:subdeg}, there is a subdegree $d$ which is a power of $p$. Then $\gcd(v-1,d)=q$. Let $f(q)=q$. Then by Lemma~\ref{lem:six}, $k$ must divide $\lambda f(q)$, and hence by \eqref{eq:m}, we have that
    \begin{align}\label{eq:parab-m}
        m<q,
    \end{align}
    Moreover, \eqref{eq:k-lam} implies that
    \begin{align}
        k &= mw(q)+1 \label{eq:parab-k} \text{ and } \\
        \lambda &=  m^2t(q)+\frac{m^2+m}{q},\label{eq:parab-lam}
    \end{align}
    where $mk=\lambda f(q)$ as in \eqref{eq:k-f} and $t(q)=(v-q-1)/q^{2}$ is a polynomial in terms of $q$. Since $\lambda$ is a positive integer, it follows form \eqref{eq:parab-lam} that $q$ divides $m^2+m$. Note that $q$ is a prime power number. Then by \eqref{eq:parab-m}, we conclude that $m=q-1$. Now \eqref{eq:parab-k} implies that $k=(q-1)w(q)+1$. In each case, we can find a polynomial $l(q)$ which is a multiple of $\gcd(k, |H\cap X|)$ satisfying $k/l(q)>q^4$. On the other hand, by Lemmas~\ref{lem:six}(b) and~\ref{lem:New}, $k$ divides $|\Out(X)|\cdot |H\cap X|$, and so we conclude that $k/l(q)$ divides $|\Out(X)|$, and since $k/l(q)>q^4$, then $|\Out(X)|>q^4$, which is impossible. For example, suppose that $X=F_{4}(q)$ and $H\cap X=P_{1}$. Then $v=(q^8-1)(q^{12}-1)(q^4-1)^{-1}(q-1)^{-1}$, and so $v-1=q\cdot w(q)$, where $w(q)=q^{10}+q^{9}+q^{8}+q^{7}+q^{6}+q^{5}+q^{4}+q^{3}+\sum_{i=0}^{14} q^{i}$. Thus $k=q^{3}(q^{12}+q^{8}-1)$, and hence $l(q)=\gcd(k,|H\cap X|)=q^{3}$. Therefore, $k/l(q)=q^{12}+q^{8}-1>q^{4}$.
\end{proof}

\subsection{Remaining cases}\label{sec:non-prab}

In this section, in order to prove Theorem~\ref{thm:main}, we need to consider the remaining large maximal subgroups of $G$ which are not parabolic, subfield and listed in Table~\ref{tbl:num}. In most cases, we follow our method which is explained in details in Subsection~\ref{sec:method} but in some cases, namely  Propositions~\ref{prop:G2} and~\ref{prop:F4}, we need extra arguments.

\begin{proposition}\label{prop:G2}
    If $X=G_{2}(q)$ and $H\cap X=SL_{3}^{\e}(q):2$ with $\e=\pm$, then $v=q^{3}(q^{3}+\e1)/2$,  $k=q^{3}(q^{3}-\e1)/6$,  and
    $\lambda=q^{3}(q^{3}-\e3)/18$, where $q=3^{a}\geq 3$.
\end{proposition}
\begin{proof}
    Suppose now that $H\cap X= SL_{3}^{\e}(q):2$ with $\e=\pm$. Then  by~\eqref{eq:v}, we have that $v=q^{3}(q^{3}+\e1)/2$, and so $v-1=v_{1}(q)/c$, where $v_{1}(q)=(q^3-1)(q^3+\e2)$ and $c=2$. Thus, Proposition~\ref{prop:g2sl3}, Proposition 1 in \cite{a:LPS2} and Lemma~\ref{lem:six}(c) implies that $k$ divides $\lambda f(q)$ where $f(q)=q^{3}-\e1$, and so there exists a positive integer $m$ such that $mk=\lambda f(q)$. By \eqref{eq:k-lam} and \eqref{eq:m}, we have that
    \begin{align}
        k&= \frac{m\cdot (q^3+\e2)+2}{2},\label{eq:G2-SL-k}\\
        \lambda&=\frac{m^2\cdot (q^{3}+\e 2)+2m}{2(q^3-\e1)},\label{eq:G2-SL-lam}
    \end{align}
    where
    \begin{align}\label{eq:G2-SL-m}
        m<q^3-\e1.
    \end{align}
    
    Note that $\lambda$ is a positive integer. Then by \eqref{eq:G2-SL-lam}, we conclude that
    \begin{align}\label{eq:G2-SL-lam-1}
        q^{3}-\e1 \text{ divides } m(3m+\e2).
    \end{align}
    Moreover, since $w(q)=v_{1}(q)/f(q)=q^3+\e2$ by \eqref{eq:w-1}, we have that
    \begin{align}\label{eq:G2-SL-k-m}
        m\cdot (q^3+\e2)+2 \text{ divides } 8ag(q),
    \end{align}
    where $g(q)=q^{3}(q^{2}-1)(q^{3}-\e 1)$. Let $h(q)= (q^2-1)(q^3-\e3)$ and $r(q)=6(q^2-1)$. Set $b_{1}:=1$ and $b:=4a$. Since $c=2$, it follows from \eqref{eq:G} that $G(m,q):=G(m,q,b,b_{1},c)=48ma(q^2-1)-16ah(q)$ with $q=p^{a}$. Therefore $G(m,q)=0$ or by \eqref{eq:G-1},
    \begin{align}\label{eq:G2-SL-k-m-1}
        m\cdot (q^3+\e2)+2 \text{ divides } |G(m,q)|.
    \end{align}
    
    \noindent \textbf{(1)} If $G(m,q)=0$, then  $48ma(q^2-1)-16ah(q)=0$, and so $m=(q^{3}-\e3)/3$. Then we must have $q=3^{a}$. By \eqref{eq:G2-SL-k} and \eqref{eq:G2-SL-lam}, we conclude that
    \begin{align*}
        k=\frac{q^{3}(q^{3}-\e1)}{6} \text{ and }
        \lambda=\frac{q^{3}(q^{3}-\e3)}{18},
    \end{align*}
    as claimed.\smallskip
    
    \noindent \textbf{(2)} If $G(m,q)>0$, then $48ma(q^2-1)>16ah(q)$, and so $m>(q^{3}-\e3)/3$, and by \eqref{eq:G2-SL-k-m-1}, we obtain $q <48a$. This inequality holds when $q=p^{a}$ is as in Table~\ref{tbl:G2-SL}. Since $(q^{3}-\e3)/3<m<q^3-\e1$, for each $q$ and $a$ as in Table~\ref{tbl:G2-SL}, we can find the value of $m$. For these values of $q$ and $m$, the statement \eqref{eq:G2-SL-k-m-1} is not true. \smallskip
    
    \begin{table}%[!htbp]
        \centering
        \scriptsize
        \caption{Possible values of $q=p^{a}$ in Proposition~\ref{prop:G2}. }\label{tbl:G2-SL}
        \resizebox{\textwidth}{!}{
            \begin{tabular}{lp{9.5cm}l}
                \noalign{\smallskip}\hline\noalign{\smallskip}
                Case & $q$ & Conditions \\
                \noalign{\smallskip}\hline\noalign{\smallskip}
                $2$ & $2^{a}$  with $a\leq 8$, $3^{a}$  with $a\leq 4$, $5^{a}$  with $a\leq 3$, $7^{a}$  with $a\leq 2$, $p$ with $p=11, 13, \ldots, 47$ & \\
                $3.2$ & $2^{a}$  with $a\leq 7$, $3^{a}$  with $a\leq 4$, $5^{a}, 7^{a}$  with $a\leq 2$, $p$ with $11, 13, 17, 19, 23$  & \\
                $3.3$ & $2^{a}$  with $a\leq 20$,
                $3^{a}$  with $a\leq 11$,
                $5^{a}$  with $a\leq 7$,
                $7^{a}$  with $a\leq 5$,
                $11^{a}, 13^{a}$  with $a\leq 4$,
                $17^{a}, 19^{a}, 23^{a}, 29^{a}$  with $a\leq 3$,
                $31^{a}, \ldots, 109^{a}$  with $a\leq 2$,
                $p$ with $113, 127, \ldots, 3067$  & $\e=-$ and $m< 32aq$ \\
                $3.3$ & $2^{a}$  with $a\leq 17$,
                $3^{a}$  with $a\leq 10$,
                $5^{a}$  with $a\leq 6$,
                $7^{a}$  with $a\leq 5$,
                $11^{a}, 13^{a}, 17^{a}, 19^{a}$  with $a\leq 3$,
                $23^{a}, \ldots, 53^{a}$  with $a\leq 2$,
                $p$ with $61, 67, \ldots, 761$ & $\e=+$ \\
                $3.3$ & $2^{a}$  with $a\leq 9$,
                $3^{a}$  with $a\leq 5$,
                $5^{a}$  with $a\leq 3$,
                $7^{a}, 11^{a}$  with $a\leq 2$,
                $p$ with $13, \ldots, 73$ & $\e=-$ and $m\geq 32aq$ \\
                \noalign{\smallskip}\hline
            \end{tabular}
        }
    \end{table}
    
    \noindent \textbf{(3)} If $G(m,q)<0$, then  $48ma(q^2-1)<16ah(q)$, and so
    \begin{align}\label{eq:G2-SL-m-2}
        m<\frac{q^{3}-\e3}{3}.
    \end{align}
    We claim that $\gcd(k, q^3)< q^2$. Assume to the contrary that $q^2$ divides $k$. Since $k-mq^3/2=\e m+1$, $q^2$ must divide $m+\e1$. Thus $m +\e1 = uq^2$ for some integer $u$. By \eqref{eq:G2-SL-m-2}, we observe that
    \begin{align}\label{eq:G2-SL-u}
        u<\frac{q}{3}.
    \end{align}
    Recall that $q^3-\e1$ divides $3m^2 +\e2m= (uq^2-\e1)(3uq^2-\e1)=3u^2(q^3-\e1)q+\e3u^2q-\e4uq^2+1$. Then
    \begin{align}\label{eq:G2-SL-fuq}
        q^3-\e1 \text{ must divide } |\e3u^2q-\e4uq^2+1|.
    \end{align}
    Let now $f_{q}^{\e}(u)=\e3u^2q-\e4uq^2+1$. For a fixed $q$, the map $f_{q}^{\e}(u)$ is decreasing (increasing) if $1\leq u\leq 2q/3$ and $\e=+$ ($\e=-$). As $u<q/3$ by \eqref{eq:G2-SL-u}, we conclude that  $|\e3u^2q-\e4uq^2+1|=|f_{q}^{\e}(u)|<|f_{q}^{\e}(q/3)|=q^{3}-\e1$, and this contradicts \eqref{eq:G2-SL-fuq}.
    
    Therefore, $\gcd(k, q^3) < q^2$, as claimed. Hence \eqref{eq:G2-SL-k} implies that
    \begin{align}\label{eq:G2-SL-k-u}
        m\cdot (q^3+\e2)+2 \text{ divides } 8aq^2(q^2-1)(q^3-\e1).
    \end{align}
    Note that $8maq^2(q^3-\e1)(q^2-1)=8ah(q)[m(q^3+\e2)+2]+G^{\e}(m,q)$, where $h(q) = q^4-q^2-\e3q$, $r(q) = 3(q^2+\e2q)$ and $G^{\e}(m,q)=\e 8mar(q)-16ah(q)$ with $q=p^{a}$ and $\e=\pm$. Then $G^{\e}(m,q)=0$ or
    we conclude by \eqref{eq:G2-SL-k-u} that
    \begin{align}\label{eq:G2-SL-k-m-2}
        m\cdot (q^3+\e2)+2 \text{ divides }  |G^{\e}(m,q)|,
    \end{align}
    
    \noindent \textbf{(3.1)} Suppose that $G^{\e}(m,q)=0$. Then $\e=+$, and so  $16ah(q)=8mar(q)$. Then $3m=2h(q)/r(q)=2q^4-4(q-1)-14/(q+2)$. This implies that $q+2$ is a multiple of $14$, and so we conclude that $q=5$ in which case $3m=1232$, which is a contradiction.\smallskip
    
    \noindent \textbf{(3.2)} Suppose that $G^{\e}(m,q)>0$. Then $\e=+$. By \eqref{eq:G2-SL-k-m-2}, $m(q^3+2)+2<|G^{\e}(m,q)|=8mar(q)-16ah(q)\leq 8mar(q)$, and so $q^3<24ar(q)$. Since $r(q) = 3(q^2+\e2q)$, it follows that $q^3<24a(q^2+2q)$, or equivalently, $q^2<24a(q+2)$. This inequality holds when $q=p^{a}$ is as in Table~\ref{tbl:G2-SL}, and for such $q$, we can obtain $m$ by \eqref{eq:G2-SL-m-2} but for these values of $(q,m)$ we cannot find any parameters satisfying \eqref{eq:G2-SL-k-m-2}.

    \noindent \textbf{(3.3)} Suppose that $G^{\e}(m,q)<0$. Let $\e=+$. Then \eqref{eq:G2-SL-k-m-2} implies that $m(q^3+2)+2<|G^{\e}(m,q)|= 16ah(q)-8mar(q)<16ah(q)$, and so $m< 16aq$. Note by \eqref{eq:G2-SL-lam-1} that $q^{3}-1$ divides $3m^{2}+2m$. Then $q^{3}-1\leq 3m^{2}+2m<3\cdot 16^{2}a^{2}q^{2}+2\cdot 16aq$, and this holds for the $q=p^{a}$ as in Table~\ref{tbl:G2-SL}. Again we can find $m$ by \eqref{eq:G2-SL-m-2} and in conclusion we cannot find any parameters satisfying \eqref{eq:G2-SL-k-m-2}. Let $\e=-$. Then \eqref{eq:G2-SL-k-m-2} implies that $m(q^3-2)+2<|G^{\e}(m,q)|= 16ah(q)+8mar(q)$. If $m\geq 32aq$, then $m[q^3-8ar(q)-2]< 16ah(q)$, and so $32aq[q^3-8ar(q)-2]< 16ah(q)$. This implies that $q<74a$. This is true for $q=p^{a}$ as in Table~\ref{tbl:G2-SL} for which there is no possible parameters satisfying \eqref{eq:G2-SL-k-m-2} when $m$  is as in \eqref{eq:G2-SL-m-2}. Therefore, $m<32aq$. Note by \eqref{eq:G2-SL-lam-1} that $q^{3}+1$ divides $3m^{2}-2m$. Then $q^{3}+1\leq m(3m-2)<32aq(3\cdot 32aq-2)$, and this holds for the $q=p^{a}$ as in Table~\ref{tbl:G2-SL}. These cases can also be ruled out as for $m$ as in \eqref{eq:G2-SL-m-2} we cannot find any parameters satisfying \eqref{eq:G2-SL-k-m-2}.
\end{proof}

\begin{proposition}\label{prop:F4}
    If $X=F_{4}(q)$, then $H\cap X$ cannot  be $2\cdot \Omega_{9}(q)$ with $q$ odd and $C_{4}(q)$ with $p=2$.
\end{proposition}
\begin{proof}
    Let $H\cap X$ be $2\cdot \Omega_{9}(q)$ with $q$ odd or $C_{4}(q)$ with $p=2$. Our argument in these cases are similar, so we only deal with the case where $H\cap X=2\cdot \Omega_{9}(q)$ with $q$ odd. Then $|H\cap X|=q^{16}(q^2-1)(q^4-1)(q^6-1)(q^{8}-1)$, and so by~\eqref{eq:v}, we have that $v=q^{8}(q^{8}+q^{4}+1)$. Then by \eqref{eq:k-g}, we have that
    \begin{align}\label{eq:F4-O9-k-d}
        k\mid bg(q),
    \end{align}
    where $b=2a$ and $g(q)=q^{16}(q^{8}-1)(q^{6}-1)(q^{4}-1)(q^{2}-1)$.
    Note that $\gcd((q^{8}-1)(q^{6}-1)(q^{4}-1)(q^{2}-1), v-1)=q^{4}+1$. By \eqref{eq:k-f}, \eqref{eq:m} and Tits' lemma~\ref{lem:Tits}, we conclude that $mk=b_{1}\lambda f(q)$ where $b_{1}=2a$, $f(q)=(q^{4}+1)$ and $m$ is a positive integer satisfying
    \begin{align}\label{eq:F4-O9-m}
        m<2a(q^{4}+1).
    \end{align}
    Therefore \eqref{eq:k-lam} implies that
    \begin{align}
        k &= \frac{m(q^{12}+q^{4}-1)}{2a}+ 1,\label{eq:F4-O9-k} \\
        4a^{2}\lambda &= m^{2}(q^{8}-q^{4}+2)-\frac{3m^{2}-2ma}{q^{4}+1}\label{eq:F4-O9-lam}
    \end{align}
    Since $\lambda$ is integer, $q^{4}+1$ divides $3m^{2}-2ma$, and so
    \begin{align}\label{eq:F4-O9-m2}
        m>\frac{q^{2}+1}{\sqrt{3}}.
    \end{align}
    
    We claim that $\gcd(k, q^{16})< q^{4}$. Assume the contrary. Then $q^{4}$ divides $2ak$. Since $mq^{4}(q^{8}+1)-[m(q^{12}+q^{4}-1)+ 2a]=m-2a$, we have
    \begin{align*}%\label{eq:F4-O9-q4}
        mq^{4}(q^{8}+1)-2ak=m-2a,
    \end{align*}
    and so $q^{4}$ must divide $m-2a$. Thus $m-2a=uq^{4}$ for some integer $u$. By \eqref{eq:F4-O9-m},
    \begin{align}\label{eq:F4-O9-u}
        u<2a.
    \end{align}
    Recall that $q^{4}+1$ divides $3m^{2}-2ma=(3q^4u^2+10a u-3u^2)(q^{4}+1)+(2a-u)(4a-3u)$. Then
    \begin{align}\label{eq:F4-1}
        q^{4}+1 \text{ must divide } (2a-u)\cdot |4a-3u|.
    \end{align}
    If $4a-3u\neq 0$, then $q^{4}+1< 8a^{2}$, which is a contradiction. If $4a-3u=0$, then $u=4a/3$, and so $m=2a(2q^{4}+3)/3$. Then by \eqref{eq:F4-O9-k}, we have that $k=q^{4}(q^{4}+1)(2q^{8}+q^{4}+1)/3$. Therefore, by \eqref{eq:F4-O9-k-d}, we conclude that $k$ divides $2ag_{1}(q)$, where  $g_{1}(q)=q^{4}(q^{8}-1)(q^{6}-1)(q^{4}-1)(q^{2}-1)$. Then, $2q^{8}+q^{4}+1$ must divide $24a(q^{6}-1)(q^{4}-1)^{2}(q^{2}-1)$, and hence $2q^{8}+q^{4}+1$ has to divide $3a(6q^{6}-17q^{4}-14q^{2}-3)$, which is impossible.
    
    Therefore, $\gcd(k, q^{16})< q^{4}$, as claimed. Now by \eqref{eq:F4-O9-k-d}, the parameter $k$ divides  $2ag_{1}(q)$, where $g_{1}(q)=q^{4}(q^{8}-1)(q^{6}-1)(q^{4}-1)(q^{2}-1)$. We continue our argument by replacing $g_{1}(q)$ with $g(q)$ in Subsection~\ref{sec:method}. Then \eqref{eq:w-1} implies that $m\cdot w(q)+ 2a$  divides $4a^2\cdot g_{1}(q)$,
    where $w(q)=q^{12}+q^{4}-1$. It follows from \eqref{eq:G-1} that
    $m(q^{12}+q^{4}-1)+ 2a \leq 4ma^2r(q)+8a^{3}h(q)$, where $h(q)= q^{12}-q^{10}-q^{8}-q^{4}+3q^{2}+2$ and $r(q)=q^{10}+q^{8}+4q^{6}+2q^{4}-3q^{2}-2$, and so \eqref{eq:F4-O9-m2} implies that
    \begin{align*}%\label{eq:F4-O9-final}
        q^{2}<\frac{(q^{2}+1)(q^{12}+q^{4}-4a^2r(q)-1)}{h(q)}\leq14a^{3}.
    \end{align*}
    As in this case $q$ is odd, this inequality implies that $q=3$ or $9$, and so by \eqref{eq:F4-O9-m}, $m$ is at most $164$ or $26248$, respectively. For each such value of $q$ and $m$, the parameters $k$ and $\lambda$ obtained in \eqref{eq:F4-O9-k} and \eqref{eq:F4-O9-lam} must be positive integers and all parameters must satisfy the conditions of symmetric designs in Lemma \ref{lem:six}, and this leads us to the parameters listed in Table~\ref{tbl:F4-O9}. However, by \eqref{eq:F4-O9-k-d}, $k$ must divide $263139026617958400=2^{16}\cdot3^{16}\cdot5^2\cdot7\cdot13\cdot41$ or $88987349938389359442577906728960000=2^{21}\cdot3^{32}\cdot5^4\cdot7\cdot13\cdot17\cdot41^2\cdot73\cdot193$, respectively for $q=3$ or $9$, which is a contradiction.
\end{proof}
\begin{table}
    \centering
    \scriptsize
    \caption{Some parameters for Proposition~\ref{prop:F4}. }\label{tbl:F4-O9}
    \begin{tabular}{llllll}
        \hline\noalign{\smallskip}
        $q$ & $a$ & $m$ & $v$ & $k$ & $\lambda$ \\
        \noalign{\smallskip}\hline\noalign{\smallskip}
        $3$ & $1$ & $28$ & $43584723$ & $7441295$ & $1270465$ \\
        $3$ & $1$ & $82$ & $43584723$ & $21792362$ & $10896181$ \\
        $9$ & $2$ & $772$ & $1853302661435043$ & $54508901806914$ & $
        1603202994321$ \\
        $9$ & $2$ & $3604$ & $1853302661435043$ & $254469018279942$ & $
        34940046551391$ \\
        $9$ & $2$ & $4376$ & $1853302661435043$ & $308977920086855$ & $
        51512015326885$ \\
        $9$ & $2$ & $13124$ & $1853302661435043$ & $926651330717522$ & $
        463325665358761$ \\
        \noalign{\smallskip}\hline
    \end{tabular}
\end{table}

%\subsection{Proofs}\label{sec:proofs}

We are now ready to prove Theorem~\ref{thm:main} and Corollaries~\ref{cor:main-1}-\ref{cor:main-2}. In what follows, we assume that $\Dmc$ is a nontrivial symmetric $(v, k, \lambda)$ design admitting a flag-transitive and point-primitive automorphism group $G$ with socle $X$ a finite simple exceptional group of Lie type.\smallskip

\noindent {\bf Proof of Theorem \ref{thm:main}} \quad
By the main result in \cite{a:Zhou-exp}, we only focus on the cases where $X$ is of type $G_{2}$, $F_{4}$, $E_{6}^{\pm}$, $E_{7}$ or $E_{8}$.

Since the group $G$ is point-primitive, the point-stabiliser $H$ is maximal in $G$, and by Corollary~\ref{cor:large}, flag-transitivity implies that $H$ is large that is to say $|G|\leq |H|^{3}$. We now apply Theorem~\ref{thm:large-ex} and analyse each possible case.

We first observe for $q=2$ that the list of maximal subgroups $H$ of $G$ can be read off from Atlas \cite{b:Atlas} and \cite{a:KW-E6-2,a:NW-F4-2}.  Note that the list of maximal subgroups of $E_6^{-}(2)$  presented in the Atlas \cite{b:Atlas} is complete (see \cite[p.304]{b:Atlas-Brauer}). We also exclude the case where $X=G_{2}(2)'$ as it is not simple. Then it is easy to check these cases by Lemma~\ref{lem:comp}  and Remark~\ref{rem:alg}, and so observe that no possible parameters sets arise in these cases.

We first conclude by Proposition~\ref{prop:num} that the numerical cases listed in Table~\ref{tbl:num} can be ruled out, and Proposition~\ref{prop:subf} shows that $H$ cannot be a subfield subgroup. If $H$ is a parabolic subgroup, then by Proposition~\ref{prop:parab}, part (a) or part (c) of Theorem~\ref{thm:main} follows.
If $X=G_{2}(q)$ and $H\cap X=SL_{3}(q)^{\e}:2$, Proposition~\ref{prop:G2} implies part (b) of Theorem~\ref{thm:main}. We note by Proposition~\ref{prop:F4} that $(X,H\cap X)\in \{(F_{4}(q),2\cdot \Omega_{9}(q)), (F_{4}(q),C_{4}(q))\}$ gives rise to no possible parameters set. We now consider the remaining possibilities for pairs $(X,H\cap X)$ as in Table \ref{tbl:large-exc-nonpar}. All these cases can be ruled out in the same manner following Steps 1-6 explained in Subsection \ref{sec:method}. Note that the required information for each case can be found in Table~\ref{tbl:rem}. Note also that as pointed out in Subsection \ref{sec:method}, the subdegrees in Theorem \ref{thm:martin} are important tools to obtain the polynomial $f(q)$ listed in Table \ref{tbl:rem}. As an example, in what follows, we show that $X$ cannot be $E_{8}(q)$.

Suppose that $X=E_{8}(q)$. Then $H\cap X$ is one of the groups
\begin{align*}
    c_{4}^{\e}\cdot (A_{2}^{\e}(q)\times E_{6}^{\e}(q)) \cdot e_{4}^{\e}\cdot 2, \  c_{1}\cdot D_{8}(q)\cdot c_{1} \text{ and }   c_{1}\cdot (A_{1}(q)\times E_{7}(q)) \cdot c_{1},
\end{align*}
where $c_{1}=\gcd(2, q-1)$, $c_{4}^{\e}=\gcd(3, q-\e1)$ and $\e=\pm$. If $H\cap X=c_{4}\cdot (A_{2}^{\e}(q)\times E_{6}^{\e}(q) )\cdot c_{4}\cdot 2$, then by~\eqref{eq:v}, we obtain parameter $v$ as in Table~\ref{tbl:rem}. For $f(q)$, $g(q)$,  $b_{1}$ and $b$ as in Table~\ref{tbl:rem}, by Table~\ref{tbl:rem}, by \eqref{eq:k-3}, we must have $(v-1)+2af(q) \leq 4a^{2}f(q)g(q)$, and so $q^{10}< 4a^{2}$, which is a contradiction. In the remaining cases, the parameter $v$, $b_{1}$ and $b$, and polynomials $f(q)$, $g(q)$, $h(q)$ are given in Table~\ref{tbl:rem}. Let $r(q)=f(q)g(q)-h(q)\cdot v_{1}(q)$, where $v-1=v_{1}(q)/c$. We note here that if $H\cap X=c_{4}\cdot D_{8}(q)\cdot c_{4}$, then $f(q)$ is the polynomial which is divisible by $\gcd(v-1,g(q))$. Since \eqref{eq:k-3} holds for almost all $q$, by \eqref{eq:v-F-2}, we have that $v<a^{2}(|r(q)|+a|f(q)h(q)|)$, and this follows that $q \in \{2,3,4,8\}$. For these values of $q=p^{a}$, we have no possible parameters set by Lemma~\ref{lem:comp}  and Remark~\ref{rem:alg}. In the case where $H\cap X=c_{1}\cdot (A_{1}(q)\times E_{7}(q) )\cdot c_{1}$ with $c_{1}=\gcd(2,q-1)$, we use the subdegrees given in Theorem~\ref{thm:martin}.
Note by Theorem~\ref{thm:martin} that the subdegrees in this case divide $c_{1}^{4}q^{32}(q^{6}+1)(q^{14}-1)(q^{18}-1)/(q^{2}-1)(q^{4}-1)$ and  $q^{28}(q^{2}-1)(q^{5}+\e1)(q^{9}+\e1)(q^{14}-1)$. As  $v-1$ and $q$ are coprime by Tits' Lemma~\ref{lem:Tits}, the greatest common divisors of $v-1$ and these subdegrees divides $f(q)=c_{1}^{2}(q^{2}-1)(q^{4}-1)$ as in Table~\ref{tbl:rem}. As \eqref{eq:k-3} holds for almost all $q$, it follows from \eqref{eq:v-F-2} that
$v< a(|r(q)|+|f(q)h(q)|)$. This implies that $q=2,3,5$ for which we can find $m$ by \eqref{eq:m}, but these values of $q$ and $m$ do not satisfy  \eqref{eq:v-F-1}, which is a contradiction.\bigskip

\noindent {\bf Proof of Corollary~\ref{cor:main-1}} \quad 
Suppose to the contrary that $\gcd(k,\lambda)=1$. We apply  Theorem~\ref{thm:main}, and observe that the possibilities (a) and (b) can be ruled out as in these cases $k$ and $\lambda$ are not coprime. Let now $X=E_{6}(q)$ with $q=p^a$ in cases (c). Then, $k=mw(q)+1$ divides $q(q^{4} +1)$, where $w(q)=q^{3}+\sum_{i=0}^{11}q^{i}$, which is impossible.

Therefore, $k$ and $\lambda$ must have at least one prime common divisor. Moreover, since $\lambda<k$, the parameter $k$ cannot be prime. Suppose now that $\lambda$ is prime. Then since $\gcd(k,\lambda)\neq 1$, we conclude that $\lambda$  divides $k$. Evidently, parts (a) and (b) of Theorem~\ref{thm:main} cannot occur. In part (c), $\lambda$ divides $6ap(q^{i}-1)$, for some $i=1,2,4,5,6,8$. Since also $\lambda$ is prime, we conclude that $\lambda$ divides $2$, $3$, $a$, $p$, or $q^{i}-1$, for some $i=1,2,3,4,5$. Thus $\lambda\in\{2,3,p\}$ or $\lambda\leq q^{4}+q^{3}+q^{2}+q+1$, and so in all cases $\lambda$ is at most $q^{4}+q^{3}+q^{2}+q+1$. Note by Lemma~\ref{lem:six}(b) that $\lambda v<k^{2}$. Since $k$ divides $q(q^{4}+1)\lambda$, we conclude that $v<q^{2}(q^{4}+1)^{2}\lambda$, and since $\lambda\leq q^{4}+q^{3}+q^{2}+q+1$, it follows that $(q^{9}-1)(q^{12}-1)<q^{2}(q-1)(q^{4}-1)(q^{4}+1)^{2}(q^{4}+q^{3}+q^{2}+q+1)$, which is impossible. Therefore, $\lambda$ is neither prime.\smallskip

\noindent {\bf Proof of Corollary~\ref{cor:main-2}}\quad 
Suppose that $\lambda \leq 100$. Then by Theorem~\ref{thm:main}, we need to consider one of the following cases:\smallskip

\noindent (a) $X=G_{2}(q)$ and $H\cap X=[q^5]:GL_{2}(q)$. Then $v=(q^{6}-1)/(q-1)$, $k=q^5$ and $\lambda=q^4(q-1)$. If $\lambda \leq 100$, then $q=2$. Note that $X=G_{2}(2)$ is not simple but for this non-simple group, we obtain $v=63$, $k=32$ and $\lambda=16$, and this is the complement of a symmetric $(63,31,15)$ design which is antiflag-transitive, see Table~\ref{tbl:cor-main} and \cite{a:Braic-2500-nopower,a:Dempwolff2001}.\smallskip

\noindent (b) $X=G_{2}(q)$ and $H\cap X=SL_{3}^{\e}(q):2$ with $\e=\pm$. Then
$v=q^{3}(q^{3}+\e1)/2$,  $k=q^{3}(q^{3}-\e1)/6$,  and $\lambda=q^{3}(q^{3}-\e3)/18$, where $q=3^{a}\geq 3$. If $\lambda\leq 100$, then $q=3$, and so \cite{a:Braic-2500-nopower,a:Dempwolff2001} implies that $G=X=G_{2}(3)$ and $H\cap X=SL_{3}^{\e}(3):2$ for $\e=\pm$ and $\Dmc$ is of parameters $(378,117,36)$  or $(351,126,45)$ respectively for $\e=+$ or $\e=-$.\smallskip

\noindent (c) $X=E_{6}(q)$ and the Levi factor of $H$ is of type $D_{5}$. Then $v=(q^{8}+q^{4}+1)(q^{9}-1)/(q-1)$ and  $k$ divides $q(q^{4}+1)\lambda$, and so by Lemma~\ref{lem:six}(b), we have that $\lambda v<k^{2}$. Therefore, $(q^{8}+q^{4}+1)(q^{9}-1)<100q^{2}(q-1)(q^{4}+1)^{2}$, which is impossible.

%\newpage

\scriptsize
\begin{longtable}{lp{9.5cm}}
    \caption{The parameters $v$, $b_{1}$ and $b$ and the polynomials $f(q)$, $g(q)$, $h(q)$ for some finite simple exceptional groups}\label{tbl:rem}\\
    Notation & $q=p^{a}$ with $p$ prime, $c_{1}=\gcd(2, q-1)$, $c_{2}^{\e}=\gcd(4, q-\e )/c_{1}$, $c_{3}^{\e}=\gcd(8,q-\e )/c_{1}$, $c_{4}^{\e}=\gcd(3,q-\e1)$, $c_{5}=\gcd(2,p)$, $c_{6}=\gcd(3,q^{2}-1)$,  $c_{7}^{\e}=\gcd(4, q-\e )$, $c_{8}^{\e}=\gcd(4,q^{4}-1)\gcd(5,q-\e1)$, $e_{1}\mid a c_{1}$, $e_{2}\mid a^{2}c_{1}c_{5}$, $e_{3}\mid ac_{5}$, $e_{4}^{\e}\mid 2ac_{4}^{\e}$\\
    \endfirsthead
    \multicolumn{2}{c}%
    {\tablename\  \thetable\ -- Continued} \\
    \noalign{\smallskip}
    \endhead
    \noalign{\smallskip}
    \multicolumn{2}{r}{-- Continued}\\
    \endfoot
    \endlastfoot
    \noalign{\smallskip}\hline\noalign{\smallskip}
    $X=G_{2}(q)$ & $H\cap X=c_{1}\cdot A_{1}(q)^2 : c_{1}$ \\
    \noalign{\smallskip}\hline\noalign{\smallskip}
    $v$ & $q^{4}(q^4+q^2+1)$ \\
    $b_{1}$ & $32a$ \\
    $b$ & $2a$ \\
    $f(q)$ & $1$ \\
    $g(q)$& $q^{2}(q^2-1)^2$  \\
    $q$  & $2, 3,4, 5, 7, 8, 9, 16, 32$ \\
    \noalign{\smallskip}\hline\noalign{\smallskip}
    $X=G_{2}(q)$ & $H\cap X=\,^{2}G_{2}(q)$ with $q=3^{2n+1}\geq 27$ \\
    \noalign{\smallskip}\hline\noalign{\smallskip}
    $v$ & $q^{3}(q^{3}-1)(q+1)$\\
    $b_{1}$ & $2a$ \\
    $b$ & $2a$  \\
    $f(q)$ & $q^{2}-q+1)$ \\
    $g(q)$&  $q^{3}(q^{3}+1)(q-1)$ \\
    $h(q)$& $q^3-3q+5$\\
    $q$  & $3^{a}$ with $a=3, 4, 5, 6, 7, 8$  \\
    \noalign{\smallskip}\hline\noalign{\smallskip}
    $X=E_{6}^{\e}(q)$ & $H\cap X=c_{1}\cdot(A_{1}(q)\times A_{5}^{\e}(q))\cdot c_{1}\cdot c_{4}^{\e}$\\
    \noalign{\smallskip}\hline\noalign{\smallskip}
    $v$ & $c_{4}^{-\e}q^{20}(q^8+q^4+1)(q^6+\e q^3+1)(q^4+1)(q^2+1)$\\
    $b_{1}$ & $1$ \\
    $b$ & $2ac_{4}^{\e}$ \\
    $f^{\e}(q)$ & $c^{\e2}_{4}c_{6}(q^{5}-\e1)(q^{2}-1)(q+\e1)$ \\
    $g^{\e}(q)$& $q^{16}(q^{6}-1)(q^{5}-\e1)(q^{4}-1)(q^{3}-\e1)(q^{2}-1)^{2}$  \\
    $h^{\e}(q)$& $ c_{4}^{\e}(q^6+\e q^5-4q^4-\e 6q^3+2q^2+\e10q+10)$,\\
    $q$ $(c_{4}^{\e}=1)$ & $2^{a}$ with $a\leq 7$, $p^{a}$ with $p=3,5,7$ and $a\leq 2$,
    $p$ with $p=11$, \ldots,$19$\\
    $q$ $(c_{4}^{\e}=3)$ & $2^{a}$ with $a\leq 14$,
    $5^{a}$ with $a\leq 5$, $7^{a}$ with $a\leq 4$, $p^{a}$ with $p=11,13,17$ and $a\leq 3$,
    $p^{a}$ with $p=19,\ldots,61$ and $a\leq 2$, $p$ with $p=67,\ldots,1933$\\
    \noalign{\smallskip}\hline\noalign{\smallskip}
    $X=E_{6}^{\e}(q)$ & $H\cap X=c_{1}^2\cdot (D_{4}(q)\times (\frac{q-\e1}{c_{1}})^2) \cdot c_{1}^2 \cdot S_{3}$ \\
    \noalign{\smallskip}\hline\noalign{\smallskip}
    $v$ & $6^{-1}c_{4}^{-\e}q^{24}(q^{12}-1)(q^{9}-\e1)(q^{8}-1)(q^{5}-\e1)(q^{4}-1)^{-2}(q-\e1)^{-2}$\\
    $b_{1}$ & $12ac_{4}^{\e}$ \\
    $b$ & $12ac_{4}^{\e}$ \\
    $f^{\e}(q)$ ($c_{4}^{\e}=1$) & $(q-\e1)^{6}(q+\e1)^{4}(q^{2}+1)^{2}$ \\
    $f^{\e}(q)$ ($c_{4}^{\e}=3$) & $2^{6}(q-\e1)^{6}$ \\
    $g(q)$ & $q^{12}(q^{6}-1)(q^{4}-1)^{2}(q^{2}-1)(q-\e1)^{2}$ \\
    \noalign{\smallskip}\hline\noalign{\smallskip}
    $X=E_{6}^{\e}(q)$ & $H\cap X=({}^3\!D_{4}(q) \times (q^2+\e q+1))\cdot 3$ with $(\e,q)\neq (-,2)$ \\
    \noalign{\smallskip}\hline\noalign{\smallskip}
    $v$ & $3^{-1}c_{4}^{-\e}q^{24}(q^{4}-1)(q^{5}-\e1)(q^{8}-1)(q^{9}-\e1) (q^{2}+\e q+1)^{-1}$ \\
    $b_{1}$ & $6ac_{4}^{\e}$ \\
    $b$ & $6ac_{4}^{\e}$ \\
    $f^{\e}(q)$ & $(q^8+q^4+1)(q^{4}+q^{2}+1)(q^2+\e q+1)$ \\
    $g^{\e}(q)$ &  $q^{12}(q^8+q^4+1)(q^6-1)(q^2-1)(q^2+\e q+1)$\\
    %   $h(q)$ & \\
    %   $r(q)$ & \\
    $q$ & $2$ for $\e=+$ \\
    \noalign{\smallskip}\hline\noalign{\smallskip}
    $X=E_{6}^{\e}(q)$ & $H\cap X=c_{7}^{\e}\cdot(D^{\e}_{5}(q)\times (\frac{q-\e1}{c_{7}^{\e}}))\cdot c_{7}^{\e}$ \\
    \noalign{\smallskip}\hline\noalign{\smallskip}
    $v$ & $c_{4}^{-\e}q^{16}(q^{12}-1)(q^{9}-\e1)(q^{4}-1)^{-1}(q-\e1)^{-1}$ \\
    $b_{1}$ & $1$ \\
    $b$ & $ 2ac_{4}^{\e}$ \\
    $f^{\e}(q)$ & $c_{8}^{\e}(q-\e1)^{2}(q^{4}+1)$ \\
    $g^{\e}(q)$ & $q^{20}(q^{8}-1)(q^{6}-1)(q^{5}-\e1)(q^{4}-1)(q^2-1)(q-\e1)$\\
    $h^{\e}(q)$& $c_{4}^{\e}(q^{20}-\e4q^{19}+5q^{18}-6q^{16}+\e7q^{15}-3q^{14}
    -\e q^{13}-2q^{12}+\e 11q^{11}-15q^{10}+\e4q^{9}+15q^{8}
    -\e24q^{7}+17q^{6}-\e4q^{5}+q^{4}-\e9q^{3}+15q^{2}-\e9q-12)$,\\
    $q$ ($c_{4}^{\e}=1$) & $2^{a}$ with $a\leq 10$, $3^{a}$ with $a\leq 8$,
    $5^{a}$ with $a\leq 4$, $p^{a}$ with $p=7,11$ and $a\leq 3$,
    $p^{a}$ with $p=13,\ldots,47$ and $a\leq 2$, $p$ with $p=53,\ldots,1301$ \\
    $q$ ($c_{4}^{\e}=3$) & $2^{a}$ with $a\leq 12$,
    $5^{a}$ with $a\leq 5$, $p^{a}$ with $p=7,11,13$ and $a\leq 4$,
    $p^{a}$ with $p=17,\ldots,31$ and $a\leq 3$, $p^{a}$ with $p=37, \ldots,151$ and $a\leq 2$, $p$ with $p=157,\ldots,11821$\\
    \noalign{\smallskip}\hline\noalign{\smallskip}
    $X=E_{6}^{\e}(q)$ & $\Soc(H)=C_{4}(q)$ with $q$ odd \\
    \noalign{\smallskip}\hline\noalign{\smallskip}
    $v$ & $c_{4}^{-\e}q^{20}(q^{5}-\e1)(q^{9}-\e1)(q^{12}-1)(q^{4}-1)^{-1}$ \\
    $b_{1}$ & $1$\\
    $b$ & $2ac_{4}^{\e}$ \\
    $f^{\e}(q)$ & $c_{4}^{\e}(q^{2}-\e c_{4}^{\e}-1)$\\
    $g(q)$ &  $q^{16}(q^2-1)(q^4-1)(q^6-1)(q^8-1)$\\
    \noalign{\smallskip}\hline\noalign{\smallskip}
    $X=E_{6}^{\e}(q)$ & $\Soc(H)=F_{4}(q)$ \\
    \noalign{\smallskip}\hline\noalign{\smallskip}
    $v$ & $c_{4}^{-\e}e_{3}^{-1}e_{4}^{\e}q^{12}(q^{9}-\e1)(q^{5}-\e1)$ \\
    $b_{1}$ & $1$ \\
    $b$ & $2ae_{3}c_{4}^{\e}$ \\
    $f(q)$ & $c_{4}^{+}$ \\
    $g(q)$ & $q^{24}(q^{2}-1)(q^{6}-1)(q^{8}-1)(q^{12}-1)$ \\
    \noalign{\smallskip}\hline\noalign{\smallskip}
    $X=E_{7}(q)$ & $H\cap X=c_{1}\cdot(A_{1}(q)\times D_{6}(q))\cdot c_{1}$\\
    %     Comments & $c_{1}=\gcd(2, q-1)$  \\
    \noalign{\smallskip}\hline\noalign{\smallskip}
    $v$ & $q^{32}(q^{18}-1)(q^{14}-1)(q^{6}+1)/(q^{4}-1)(q^{2}-1)$ \\
    $b_{1}$ & $1$\\
    $b$ & $a$ \\
    $f^{\e}(q)$ & $4c_{4}^{\e}\cdot(q^{2}-1)(q^{8}-1)$ \\
    $g(q)$ & $q^{31}(q^{10}-1)(q^{8}-1)(q^{6}-1)^{2}(q^{4}-1)(q^{2}-1)^{2}$ \\
    $h(q)$ & $4q^{15}-16q^{13}+16q^{11}+4q^{9}-4q^{5}-32q^{3}+20q$\\
    $q$ &
    $2^{a}$ with $a\leq 8$, $3^{a}$ with $a\leq 3$,
    $p^{a}$ with $p=5,7$ and $a\leq 2$,
    $p$ with $p=11$, \ldots,$31$\\
    \noalign{\smallskip}\hline\noalign{\smallskip}
    $X=E_{7}(q)$ & $H\cap X=c_{2}^{\e}\cdot (A_{7}^{\e}(q) \cdot c_{3}^{\e} \cdot  (2\times (2/c_{2}^{\e}))$ with $\e=\pm$ and $q>2$\\
    %     Comments & $c_{2}^{\e}=\gcd(4, q-\e )/\gcd(2, q-1)$, $c_{3}^{\e}=\gcd(8,q-\e )/\gcd(2, q-1)$, $\e=\pm$, $q>2$  \\
    \noalign{\smallskip}\hline\noalign{\smallskip}
    $v$ & $q^{35}(q^{18}-1)(q^{12}-1)(q^{7}+\e1)(q^{5}+\e1)/4\cdot (q^{4}-1)(q^{3}-\e1)$ \\
    $b_{1}$ & $1$ \\
    $b$ & $4a$\\
    $f(q)$ & $(q-\e1)^{3}(q+\e2)$ \\
    $g(q)$ &  $q^{28}(q^{8}-1)(q^{7}-\e1)(q^{6}-1)(q^{5}-\e1)(q^{4}-1)(q^{3}-\e1)(q^{2}-1)$\\
    \noalign{\smallskip}\hline\noalign{\smallskip}
    $X=E_{7}(q)$ & $H\cap X=c_{4}^{\e}\cdot (E_{6}^{\e}(q)\times (q-\e/c_{4}^{\e}))\cdot c_{4}\cdot 2$ with $\e=\pm$ and $(q,\e)\neq (2,-)$ \\
    %    Comments & $c_{4}^{\e}=\gcd(3,q-\e1)$, $c_{1}=\gcd(2,q-1)$, $\e=\pm$ and $(q,\e)\neq (2,-)$\\
    \noalign{\smallskip}\hline\noalign{\smallskip}
    $v$ & $q^{27}(q^5+\e1)(q^9+\e1)(q^{14}-1)/[2c_{1}\cdot (q-\e)]$ \\
    $b_{1}$ & $1$ \\
    $b$ & $2a$\\
    $f^{\e}(q)$ & $c_{4}^{\e}\cdot (q^{9}-\e1)$\\
    $g^{\e}(q)$ & $2q^{36}(q-\e1)(q^2-1)(q^5-\e1)(q^6-1)(q^8-1)(q^9-\e1)(q^{12}-1)$ \\
    $h^{\e}(q)$&
    $c_{1}\cdot (4q^{34}-\e8q^{33}+\e8q^{31}-4q^{30}-\e8q^{29}+12q^{28}+\e
    8q^{27}-20q^{26}-\e4q^{25}+36q^{24}-\e16q^{23}-40
    q^{22}+\e44q^{21}+24q^{20}-\e76q^{19}+4q^{18}+\e92q^{17}
    -52q^{16}-\e84q^{15}+112q^{14}+\e56q^{13}
    -168q^{12}+\e24q^{11}+192q^{10}-\e124q^{9}-164q^{8}
    +\e228q^{7}+72q^{6}-\e312q^{5}+64q^{4}+\e320q^{3}
    -236q^{2}-\e240q+384)$,\\%
    $q$ for $\e=+$ & $2^{a}$ with $a\leq 12$, $3^{a}$ with $a\leq 7$,
    $5^{a}$ with $a\leq 4$, $7^{a}$ with $a\leq 4$,
    $13^{a}$ with $a\leq 3$,
    $p^{a}$ with $p=11$, $17$, \ldots,$43$ and $a\leq 2$,
    $p$ with $p=47$, \ldots,$997$\\
    $q$ for $\e=-$ &  $2^{a}$ with $a\leq 11$,
    $3^{a}$ with $a\leq 7$,
    $5^{a}$ with $a\leq 5$,
    $7^{a}$ with $a\leq 3$, $11^{a}$ with $a\leq 3$,
    $p^{a}$ with $p=13$,\ldots, $23$ and  $a\leq 2$,
    $p$ with $p=29$, \ldots, $1013$ \\
    \noalign{\smallskip}\hline\noalign{\smallskip}
    $X=E_{7}(q)$ & $\Soc(H)=A_{1}(q)\times F_{4}(q)$ with $q>3$ \\
    %    Comments & $c_{1}=\gcd(2,q-1)$, $e_{1}\mid c_{1}a$ and $e_{2}\mid c_{1}c_{5}a^{2}$, $c_{5}=\gcd(2,p)$\\
    \noalign{\smallskip}\hline\noalign{\smallskip}
    $v$ & $e_{1}q^{36}(q^4+1)(q^2+1)(q^{12}-1)/(c_{1}e_{2})$ \\
    $b_{1}$ & $2a^{2}$ \\
    $b$ & $2a^{2}$ \\
    $f(q)$ &  $c_{1}^{3}e_{2}^{3}(q^2+1)^2(q^4+1)(q^4-q^2+1)$\\
    $g(q)$ & $q^{25}(q^2-1)^2(q^6-1)(q^8-1)(q^{12}-1)$ \\
    \noalign{\smallskip}\hline\noalign{\smallskip}
    $X=E_{8}(q)$ & $H\cap X=c_{1}\cdot D_{8}(q)\cdot c_{1}$ \\
    \noalign{\smallskip}\hline\noalign{\smallskip}
    $v$ & $q^{64}(q^{30}-1)(q^{24}-1)(q^{20}-1)(q^{18}-1)/(q^{10}-1)(q^{8}-1)(q^{6}-1)(q^{4}-1)$\\
    $b_{1}$ & $a$\\
    $b$ & $a$\\
    $f(q)$ & $(q^{2}-1)^{8}(q^{2}+1)^{4}(q^{4}+1)^{2}(q^{14}-1)$\\
    $g(q)$ & $q^{56}(q^{8}-1)\prod_{i=1}^{7}(q^{2i}-1)$\\
    $h(q)$ & $q^{36}-4q^{34}+2q^{32}+10q^{30}-14q^{28}+6q^{26}+2q^{24}-30q^{22}+42q^{20}+ 3q^{18}-28q^{16}+38q^{14}-44q^{12}-14q^{10}+41q^{8}-33q^{6}+16q^{4}+30q^{2}-3$\\
    \noalign{\smallskip}\hline\noalign{\smallskip}
    $X=E_{8}(q)$ & $H\cap X=c_{1}\cdot (A_{1}(q)\times E_{7}(q) )\cdot c_{1}$ \\
    \noalign{\smallskip}\hline\noalign{\smallskip}
    $v$ & $q^{56}(q^{30}-1)(q^{24}-1)(q^{20}-1)/(q^{10}-1)(q^{6}-1)(q^{2}-1)$\\
    $b_{1}$ & $1$ \\
    $b$ & $a$ \\
    $f(q)$ & $c_{1}^{2}(q^{14}-1)(q^{2}-1)$\\
    $g(q)$ & $q^{64}(q^{18}-1)(q^{14}-1)(q^{12}-1)(q^{10}-1)(q^{8}-1)(q^{6}-1)(q^{2}-1)^{2}$\\
    $h(q)$ & $c_{1}^{2}(q^{40}-4q^{38}+6q^{36}-6q^{34}+8q^{32}-10q^{30}+10q^{28}-10q^{26}+11q^{24}-15q^{22}+18q^{20}-20q^{18}+22q^{16}-22q^{14}+23q^{12}-19q^{10}+16q^{8}-17q^{6}+14q^{4}-12q^{2}+8)$, \\
    \noalign{\smallskip}\hline\noalign{\smallskip}
    $X=E_{8}(q)$ & $H\cap X=c_{4}^{\e}\cdot (A_{2}^{\e}(q)\times E_{6}^{\e}(q) )\cdot c_{4}^{\e}\cdot 2$\\
    \noalign{\smallskip}\hline\noalign{\smallskip}
    $v$ & $2^{-1}q^{81}(q^{30}-1)(q^{24}-1)(q^{20}-1)(q^{18}-1)(q^{14}-1)/(q^{9}-\e)(q^{6}-1)(q^{5}-\e)(q^{3}-\e)(q^{2}-1)$\\
    $b_{1}$ & $2a$ \\
    $b$ & $2a$ \\
    $f(q)$ & $(q^{12}-1)(q^{9}-\e)(q^{8}-1)(q^{6}-1)(q^{5}-\e)(q^{3}-\e)(q^{2}-1)^{2}$\\
    $g(q)$ & $q^{39}(q^{12}-1)(q^{9}-\e)(q^{8}-1)(q^{6}-1)(q^{5}-\e)(q^{3}-\e)(q^{2}-1)^{2}$\\
    \noalign{\smallskip}\hline\noalign{\smallskip}
\end{longtable}
\normalsize

\section*{Acknowledgments}

The authors would like to thank anonymous referees for providing us helpful and constructive comments and suggestions. The authors are also grateful to Martin Liebeck for Theorem~\ref{thm:martin} and Proposition~\ref{prop:g2sl3}. The first and third author are grateful to Alice Devillers  Cheryl E. Praeger and John Bamberg for supporting their visit to UWA (The University of Western Australia) during March-April 2017 and July-September 2019. The first author would like to thank IPM (Institute for Research in Fundamental Sciences). Part of this investigation was supported by a grant from IPM (N.94200068).%This paper is part of the Ph.D. thesis of the second author.

%\section*{References}

%\section*{References}
%\bibliographystyle{elsart-num-sort}
%\bibliographystyle{abbrv}
%\bibliography{/Users/Alavi/Dropbox/References/references.bib}
%\bibliography{D:/Dropbox/References/references}

\end{document}